\documentclass[english]{amsart}

\usepackage{amsmath}
\usepackage{amssymb}
\usepackage{amsfonts}
\usepackage{amsthm}
\usepackage{mathrsfs}
\usepackage[all]{xy}
\usepackage[pdftex]{graphicx}
\usepackage{rotating}
\usepackage{tikz-cd}
\usepackage{hyperref}

\newcommand{\ad}{\operatorname{ad}}

\newcommand{\Hom}{\operatorname{Hom}}

\newcommand{\End}{\operatorname{End}}
\newcommand{\Frob}{\operatorname{Frob}}

\newcommand{\ext}{\operatorname{Ext}}

\newcommand{\id}{\operatorname{id}}
\newcommand{\Ind}{\operatorname{Ind}}
\newcommand{\N}{\operatorname{N}}
\newcommand{\ind}{\operatorname{ind}}

\newcommand{\isom}{\stackrel{\sim}{\rightarrow}}

\newcommand{\im}{\operatorname{im}}

\newcommand{\SL}{\mathrm{SL}}

\newcommand{\GL}{\mathrm{GL}}
\newcommand{\R}{\mathrm{R}}

\newcommand{\val}{\mathrm{val}}

\newcommand{\surj}{\twoheadrightarrow}
\newcommand{\inj}{\hookrightarrow}

\newcommand{\comment}[1]{}
\renewcommand{\(}{\left(}
\renewcommand{\)}{\right)}

\newtheorem{thm}{Theorem}[section]
\numberwithin{thm}{subsection}
\newtheorem{lem}[thm]{Lemma}

\newtheorem{prop}[thm]{Proposition}
\newtheorem{corr}[thm]{Corollary}
\newtheorem{ax}[thm]{Axiom}
\theoremstyle{remark}
\newtheorem{remark}[thm]{Remark}
\theoremstyle{definition}

\newtheorem{example}[thm]{Example}

\numberwithin{equation}{subsection}

\def\NN{\mathbf N}
\def\ZZ{\mathbf Z}
\def\QQ{\mathbf Q}

\def\FF{\mathbf{F}}

\def\AA{\mathbf{A}}

\def\GG{\mathbf{G}}
\def\BB{\mathbf{B}}
\def\TT{\mathbf{T}}
\def\UU{\mathbf{U}}

\def\rarrow{\rightarrow}

\def\cS{\mathcal{S}}
\def\cO{\mathcal{O}}
\def\cG{\mathcal{G}}
\def\cM{\mathcal{M}}
\def\cE{\mathcal{E}}
\def\cH{\mathcal{H}}
\def\rbar{\overline{r}}
\def\rhobar{\overline{\rho}}
\def\psibar{\overline{\psi}}
\def\pibar{\overline{\pi}}
\def\taubar{\overline{\tau}}
\def\chibar{\overline{\chi}}
\def\etabar{\overline{\eta}}

\def\Pbar{\overline{P}}

\def\Qbar{\overline{\mathbf{Q}}}

\def\tld{\widetilde}

\DeclareMathAlphabet{\mathpzc}{OT1}{pzc}{m}{it}

\title{Lattices in the cohomology of $U(3)$ arithmetic manifolds}
\author{Daniel Le}
\begin{document}

\begin{abstract}
Under hypotheses required for the Taylor--Wiles method, we prove for forms of $U(3)$ which are compact at infinity that the lattice structure on upper alcove algebraic vectors or on principal series types given by the $\lambda$-isotypic part of completed cohomology is a local invariant of the Galois representation attached to $\lambda$ when this Galois representation is residually irreducible locally at places dividing $p$.
As a crucial input, we establish corresponding mod $p$ multiplicity one results.
Our main innovation is the combination of integral Hecke theory and the Taylor--Wiles method.
\end{abstract}

\maketitle

\section{Introduction}

\subsection{Background and the main result}

Our goal is to prove a $p$-adic local-global compatibility result.
Let $K/\QQ_p$ be a finite extension and $\rho:G_K \rarrow \GL_n(\Qbar_p)$ a regular de Rham Galois representation.
Let $\pi$ be the smooth $\GL_n(K)$-representation corresponding to the Weil--Deligne representation $\mathrm{WD}(\rho)$ via the local Langlands correspondence, and let $V$ be the algebraic $\GL_n(K)$-representation corresponding to the Hodge--Tate weights of $\rho$.
If $n=2$ and $K=\QQ_p$, the $p$-adic Langlands correspondence for $\GL_2(\QQ_p)$, initiated in \cite{Breuilstart} and established in  \cite{MR2642409,MR3150248,CDP}, attaches to $\rhobar$ a $p$-adic Banach space completion of $\pi \otimes V$.
Moreover, the correspondence satisfies \emph{$p$-adic local-global compatibility} by \cite{Emerton}, that is, the above completion coincides with the completion found in the $r$-part of completed cohomology of modular curves if $r:G_\QQ \rarrow \GL_2(\Qbar_p)$ is a modular Galois representation whose restriction $r|_{D_p}$ is isomorphic to $\rho$.

If $n>2$ or $K\neq\QQ_p$, there is increasing evidence, though no definitive conjecture, for an analogous correspondence for $\GL_n(K)$ (see \cite{MR2827792}).
However, if $\rho$ is the local restriction $r|_{D_v}$ of a modular Galois representation $r:G_F \rarrow \GL_n(\Qbar_p)$ of a CM field $F$, then the completed cohomology of $U(n)$ arithmetic manifolds gives a natural integral structure, and thus a norm, on $\pi \otimes V$.
Though this norm is of global origin, it is natural to hope for $p$-adic local-global compatibility results relating this norm to $\rho$, and thus informing our understanding of the hypothetical correspondence.
In recent years, significant progress on $p$-adic local-global compatibility has been made in the case when $n=2$ and $K/\QQ_p$ is an unramified extension, most notably the proof in \cite{1305.1594} of \cite[Conjecture 1.2]{breuil}.
If $\cO_K$ denotes the ring of integers of $K$, then one can study $\GL_n(K)$-representations by studying certain $\GL_n(\cO_K)$-subrepresentations called $\GL_n(\cO_K)$-types.
This is a powerful tool in the smooth representation theory of $p$-adic groups (see \cite{Bushnell}), and it is natural to import it to the study of $p$-adic representations.
\cite{1305.1594} proves a $p$-adic local-global compatibility result for the integral structure on generic tame $\GL_n(\cO_K)$-types.
We prove analogous results for generic tame principal series $\GL_3(\ZZ_p)$-types in supersingular cases, to our knowledge, the first local-global compatibility results in the $p$-adic Langlands program for a group of semisimple rank greater than one.

We now explain our generalization in more detail.
We now take $n$ to be $3$.
Let $F$ be a CM field with totally real subfield $F^+$ and assume that $F/F^+$ is unramified at all finite places.
Suppose that $p\neq 2$ splits completely in $F$.
Let $E$ be a finite extension of $\QQ_p$ with ring of integers $\cO_E$, uniformizer $\varpi_E$, and residue field $k_E$ such that $r$ is defined over $E$.
Suppose for simplicity that $r$ is unramified outside of $p$ (see Theorems \ref{mainalg} and \ref{mainps} for weaker hypotheses).
Let $\lambda$ be the Hecke eigensystem corresponding to $r$.
Consider an integral model $\UU_{/\cO_{F^+}}$, which is reductive at all split places, for a form of $U(3)$ over $F^+$ that splits over $F$, is quasisplit at all finite places, and is compact at infinity. 
Let $\widehat{H}^0$ be the completed cohomology with $\mathcal{O}_E$-coefficients of the associated arithmetic manifold (with tame level as described in \S \ref{aut}) and assume that $\widehat{H}^0[\lambda] \neq 0$.
For each place $v|p$ of $F^+$, choose a place $\tld{v}|v$ of $F$.
Let $V_{\tld{v}}$ be the algebraic representation corresponding to the Hodge--Tate weights of $r|_{D_{\tld{v}}}$ and $\tau_{\tld{v}} \subset \pi_{\tld{v}}$ be the $\GL_3(\cO_{F_{\tld{v}}})$-type and $\GL_3(F_{\tld{v}})$-representation, respectively, corresponding to $\mathrm{WD}(r|_{D_{\tld{v}}})$ by the (inertial) local Langlands correspondence as in \cite[Proposition 6.5.3]{MR2656025}.
%Let $V_v$ and $\tau_v$ be the $\UU(\cO_{F^+_v})$-representations corresponding to $V_{\tld{v}}$ and $\tau_{\tld{v}}$ under the natural isomorphism $\iota_{\tld{v}}: \UU(F_v^+) \isom \GL_3(F_{\tld{v}})$.
Since
\[\Hom(\otimes_{v|p} \pi_{\tld{v}}\otimes V_{\tld{v}}, \widehat{H}^0[\lambda] \otimes_{\cO_E} E) \cong \Hom(\otimes_{v|p} \tau_{\tld{v}} \otimes V_{\tld{v}},\widehat{H}^0[\lambda] \otimes_{\cO_E} E)\]
(as $\prod_{v|p} \UU(F^+_v)$ and $\prod_{v|p} \UU(\cO_{F^+_v})$-representations, respectively, under the natural isomorphism $\iota_{\tld{v}}: \UU(F_v^+) \isom \GL_3(F_{\tld{v}})$) is one-dimensional over $E$ by \cite[Theorems 5.4 and 5.9]{MR2856380},
the intersection $(\otimes_{v|p} \tau_{\tld{v}} \otimes V_{\tld{v}} \otimes_{\QQ_p} E) \cap \widehat{H}^0[\lambda]$ gives a homothety class of lattices in $\otimes_{v|p} \tau_{\tld{v}}\otimes V_{\tld{v}} \otimes_{\QQ_p} E$.
In this paper we will study the following two situations: either $r|_{D_{\tld{v}}}$ is crystalline with upper alcove Hodge--Tate weights, in which case $\tau_{\tld{v}}$ is trivial and $V_{\tld{v}}$ is what we will call an upper alcove algebraic representation (see \S \ref{notation}); or
else $r|_{D_{\tld{v}}}$ is potentially crystalline of principal series Galois type, in which case $\tau_{\tld{v}}$ is a
principal series type and $V_{\tld{v}}$ is trivial.
The following is our main result.

\begin{thm} \label{mainthm}
Suppose that for all places $v|p$ of $F^+$, $r|_{D_{\tld{v}}}$ is crystalline of Hodge--Tate weights $(c_{\tld{v}}+p+1,b_{\tld{v}}+1,a_{\tld{v}}-p+1)$ or potentially crystalline at all places $v|p$ of Hodge--Tate weights $(0,1,2)$ and type $\eta^{-a_{\tld{v}}} \oplus \eta^{-b_{\tld{v}}} \oplus \eta^{-c_{\tld{v}}}$.
Suppose that $\rbar$ satisfies the Taylor--Wiles conditions, 
%the product of the deformation spaces away from $p$ with the types of $r$ is regular, 
$\rbar|_{G_{F_{\widetilde{v}}}}$ is irreducible for places $v|p$, and that $a_{\tld{v}}-b_{\tld{v}}>6$, $b_{\tld{v}}-c_{\tld{v}}>6$, and $a_{\tld{v}}-c_{\tld{v}}<p-5$.
Then the homothety class of the lattice $(\otimes_{v|p} \tau_{\tld{v}} \otimes V_{\tld{v}} \otimes_{\QQ_p} E) \cap \widehat{H}^0[\lambda]$ in $\tau_{\tld{v}}\otimes V_{\tld{v}} \otimes_{\QQ_p} E$ can be described in terms of the crystalline Frobenius eigenvalues of $r|_{D_{\tld{v}}}$ $($see Theorems \ref{mainalg} and \ref{mainps} for the precise description$)$.
\end{thm}

Note that in the crystalline case with lower alcove Hodge--Tate weights, there is only one lattice up to homothety, and so the analogous result is trivial.
In the $\GL_2(\QQ_p)$ crystalline case, an analogous relationship between lattices in algebraic vectors and Hecke eigenvalues is used in \cite{BG,BG13,BG15,BGR} to compute the reduction of two-dimensional crystalline representations.
See Corollary \ref{corr:val} for a partial result in this direction for $\GL_3(\QQ_p)$.
In the unramified $\GL_2$ tame principal series case, \cite{BP,breuil} use analogous results to study the representations occuring in the mod $p$ cohomology of Shimura curves.
It is natural to hope for a similar theory for the $U(3)$ arithmetic manifolds that we study.

\subsection{Methods and an overview of the paper} \label{overview}

The first step, which is the subject of \S \ref{sec:lattice}, is to classify lattices in $\tau_{\tld{v}}\otimes V_{\tld{v}}$.
For the relevant algebraic representations, this follows from \cite[Part II]{MR2015057}.
For generic tame principal series types, there are partial classifications of lattices in the literature, e.g. in \cite[\S 8]{MR0396731} and  \cite{MR759299}, using classical intertwiners for principal series.
To get a full classification, it turns out to be sufficient to compute the submodule structure of the reduction of natural lattices in principal series types, which we compute using results of \cite{MR2845621}.
The next step, which is the subject of \S \ref{hecke}, is to isolate certain integral Hecke operators and show how they determine invariants of lattices.
The computation of these integral Hecke operators can be seen as an extension of some results of \cite{MR2845621}.
Finally, we combine our representation theoretic results with recent modifications of the Taylor--Wiles method (\cite{MR1333036}).
Kisin's local-global modification (\cite{MR2600871}) provides a natural setting to compare the local and global Langlands correspondences integrally and in families.
More specifically, we prove Theorem \ref{mainthm} by studying a patching functor applied to maps between lattices.
The modification in \cite{1310.0831} allows us to describe these maps in terms of Hecke operators.
We describe the necessary results of \cite{1310.0831} in \S \ref{patching}.
%The use of representation theory is essential because for $\GL_n$ with $n>2$, there are irreducible mod $p$ representations of $\GL_n(\FF_p)$ that do not lift to characteristic zero and so have no known Galois theoretic interpretation.
In \S \ref{cyclic}, we take the geometric perspective of \cite{MR3134019} and \cite{1305.1594} and prove Theorem \ref{mainthm} in families where one can use that patched modules are Cohen--Macaulay.
The Cohen--Macaulay property is used in two essential ways.
First, we use our representation theoretic results to augment (see Theorems \ref{cyclicalg} and \ref{cyclicps}) the argument of \cite{MR1440309,math/0602606} that uses the Auslander--Buchsbaum formula to prove mod $p$ multiplicity one results.
Second, we use that a generically vanishing submodule of a Cohen--Macaulay module is in fact zero to show that a single Hecke operator determines lattices in the crystalline case.
Finally, in \S \ref{lattice}, we use classical local-global compatibility to compute Hecke eigenvalues, and thus lattices in cohomology, from crystalline Frobenius eigenvalues of $r|_{D_{\tld{v}}}$.

\subsection{Acknowledgments}

This work grew out of my thesis.
I am deeply indebted to my adviser Matthew Emerton for guidance, support, and numerous conversations about patching and the $p$-adic Langlands program.
I thank Florian Herzig for correspondence regarding Proposition \ref{submodule}.
I thank Christophe Breuil for sharing a letter about lattice conjectures for $\GL_3(\QQ_p)$ which in addition to \cite{breuil}
%and \cite{1208.5367} 
inspired this paper.
This paper also owes a debt to ideas (some unpublished and some in \cite{MR3079258,1305.1594,1310.0831}) of Ana Caraiani, Matthew Emerton, Toby Gee, David Geraghty, Florian Herzig, Vytautas Paskunas, David Savitt, and Sug Woo Shin, and it is a pleasure to acknowledge this.
We thank Matthew Emerton, Toby Gee, Florian Herzig, and Bao V. Le Hung for comments on previous drafts of this paper.
Finally, we heartily thank the referee for several comments which improved the exposition and accuracy of the paper.

\subsection{Notation} \label{notation}

For a field $k$, $G_k$ denotes the absolute Galois group of $k$. We denote the cyclotomic character by $\epsilon$. Hodge--Tate weights are normalized so that $\epsilon$ has weight $-1$.
We denote by $F^+ \subset F$ the maximal totally real subfield of a CM field.
The letter $v$ is used to denote places of $F^+$ while the letter $w$ is used to denote places of $F$.
Further, $\tld{v}$ denotes a place of $F$ dividing a place $v$ of $F^+$.
We use $E$ to denote a finite extension of $\QQ_p$ with ring of integers $\cO_E$, uniformizer $\varpi_E$, and residue field $k_E$.

Compact induction is denoted $\ind_H^G$, while usual induction is denoted $\Ind_H^G$. The symbols $(\cdot)^\vee$ and $(\cdot)^d$ are used to denote the Pontriagin dual and Schikhof dual, respectively (see \cite[\S 1.8]{1310.0831}). The symbol $(\cdot)^*$ is used to denote the contragredient of a representation.

Let $\BB$ and $\TT\subset {\GL_n}_{/\ZZ}$ be the algebraic subgroups of upper triangular and diagonal matrices, respectively.
As usual, the group of characters (resp.~cocharacters) of $\TT$ are denoted by $X^*(\TT)$ (resp.~$X_*(\TT)$).
These groups are identified with $\ZZ^n$ in the usual way.
Further subscripts with $+$ (resp.~$-$) denote dominant (resp.~antidominant) characters and cocharacters.
For a dominant character $\nu \in X^*(\TT)_+$, we say that $\nu$ is $p$-restricted if $0\leq \langle \nu,\alpha^\vee\rangle<p$ for all simple positive roots $\alpha$.
If $\nu\in X^*(\TT)_+$ is a dominant and $p$-restricted character, let $W(\nu)$ be the $\ZZ_p$-points of the algebraic $\GL_n$-representation $\Ind_\BB^{\GL_n} w_0 \nu$, where $w_0$ denotes the longest element of $S_n$.
Let $\overline{W}(\nu)$ be the reduction of $W(\nu)$ modulo $p$, and let $F(\nu)$ be the (irreducible) socle of $\overline{W}(\nu)$.
These characteristic $p$ representations factor through $\GL_n(\FF_p)$.
Every (absolutely) irreducible representation of $\GL_n(\FF_p)$ is isomorphic to $F(\nu)$ for some $\nu$, and the representations $F(\mu)$ and $F(\nu)$ are isomorphic if and only if $\mu \equiv \nu \pmod{p-1}$ (see \cite[Theorem 3.10]{herzig}).
The representations $F(\nu)$ are called (\emph{Serre}) \emph{weights} (of $\GL_n(\FF_p)$).

Usually, we will take $n$ above to be $3$.
We label the elements of $S_3$ as $e=()$, $s_1 = (12)$, $s_2 = (23)$, $r_1 = (123)$, $r_2 = (132)$, and $w_0 = (13)$.
The ordered pair $(i,j)$ often denotes some permutation of the ordered pair $(1,2)$.
Let $(x,y,z) \in X^*(\TT)_+$.
The character $(x,y,z)$ is in the \emph{lower alcove} (resp.~\emph{upper alcove}) if $x-z<p-2$ (resp.~$x-y,y-z<p-1$ and $p-2<x-z$) (see the paragraph following \cite[Corollary 4.8]{herzigthesis}).
We say that $F(\nu)$ is a lower (resp.~upper) alcove Serre weight if $\nu$ is a lower (resp.~upper) alcove character.

The triple $(a,b,c)$ with $a>b>c$ will be used to denote the highest weight of an algebraic representation of $\GL_3(\FF_p)$.
We will assume throughout (with the exception of \S \ref{latps}) that $a-c < p$ and usually that $a-c < p-1$.
In \S \ref{lattice}, we assume that $a-b>6$, $b-c>6$, and $a-c< p-5$ (the {\it strong genericity} hypothesis of \cite{MR3079258}) to apply \cite[Theorem 7.5.5]{MR3079258}.

By \cite[Theorem 2.8 and \S 4]{MR916175}, we have the following results.
Note that \cite{MR916175} describes $\SL_3(\FF_p)$-representations, but the proofs carry over verbatim to the $\GL_3(\FF_p)$-setting.
Alternatively, the results in the $\GL_3(\FF_p)$-setting can be easily deduced from the $\SL_3(\FF_p)$-setting via the inflation-restriction exact sequence.
If $\sigma$ and $\sigma'$ are weights, then 
\[\dim_{\FF_p} \ext^1_{\GL_3(\FF_p)}(\sigma',\sigma) = \dim_{\FF_p} \ext^1_{\GL_3(\FF_p)}(\sigma,\sigma') \leq 1.\]
If $F(x,y,z)$ is a weight, let $\cE(x,y,z)$ be the set of triples $(a,b,c)$ such that 
\[\ext^1_{\GL_3(\FF_p)}(F(x,y,z),F(a,b,c)) \neq 0.\]
If $(x,y,z)$ is a lower alcove character, then 
\begin{align*}
\cE(x,y,z) := \{&(y+p-1,x,z), (x,z,y-p+1), (z+p-2,y,x-p+2), \\
&(x+1,z-1,y-p+1), (y+p-1,x+1,z-1),(x,z-1,y-p+2), \\
&(y+p-2,x+1,z) \}.
\end{align*}
Moreover, $\ext^1_{\GL_3(\FF_p)}(\sigma',\sigma) = 0$ if both $\sigma$ and $\sigma'$ are upper alcove weights.
%If $(x,y,z) \in \cE(a,b,c)$, let $E((a,b,c),(x,y,z))$ be the unique up to isomorphism nontrivial extension of $F(a,b,c)$ by $F(x,y,z)$.
%In \S \ref{heckealg} and \ref{keyiso}, $W_1$ and $W_2$ are chosen to be certain nonsplit extensions of weights, for example $\overline{W}(\nu)$ for $\nu$ an upper alcove weight.

\section{Lattices in locally algebraic representations} \label{sec:lattice}
Let $\tau$ be either an upper alcove algebraic representation or a sufficiently generic principal series
type over E.
Then $\tau$ is residually multiplicity free in the sense that the semisimplification $\taubar^{\mathrm{ss}}$ of the reduction of any lattice in $\tau$ is multiplicity free.
By \cite[Lemma 4.1.1]{1305.1594}, this condition implies that for each Jordan--H\"older factor $\sigma$ of $\taubar^{\mathrm{ss}}$, there is up to homothety a unique lattice $\tau^\sigma$ in $\tau$ with cosocle isomorphic to $\sigma$.
In this section we make a detailed study of the lattices $\tau^\sigma$ (and the relations between them for varying $\sigma$) for upper alcove algebraic representations (\S \ref{algvect}) and principal series types (\S \ref{latps}). In \S \ref{sec:morita}, we use Morita theory to explain how this information leads to a classification of lattices in these representations.

\subsection{Lattices in algebraic vectors} \label{algvect}

In this section, we define two natural integral structures on upper alcove algebraic vectors.
We follow the notation of \cite[\S 3]{herzig}.
Let $K$ be $\GL_3(\ZZ_p)$ and $K_1$ be the kernel of the reduction map $\GL_3(\ZZ_p) \surj \GL_3(\FF_p)$.

Fix $a>b>c$ integers such that $a-c<p$. Note that $(a-1,b,c+1)$ is a lower alcove character.
Let $\nu$ be the ($p$-restricted) upper alcove character $(c+p-1,b,a-p+1)$.
Throughout this paper, we will denote $W(\nu)$ and $W(-w_0 \nu)^*$ by $W$ and $W^0$, respectively, where $\cdot^*$ denotes the contragredient representation.
By Serre duality for $\GL_3/\BB$, there is an isomorphism $W^0 \cong\R^3 \Ind_\BB^{\GL_3} (\nu - 2\rho)(\ZZ_p)$ where $2\rho$ is the sum of the positive roots of $\GL_3$ (see \cite[Chapter II.4(8)]{MR2015057}).

Let $V = W[p^{-1}] = \Ind_\BB^{\GL_3} w_0 \nu(\QQ_p)$.
By the Borel-Weil-Bott theorem, $V \cong \R^3 \Ind_\BB^{\GL_3} (\nu - 2\rho)(\QQ_p) \cong W^0[p^{-1}]$ (see \cite[II.5.3]{MR2015057}), and so $W^0$ is (isomorphic to) a lattice in $V$.
The next proposition shows that the lattices $W$ and $W^0$ are the two lattices up homothety with irreducible cosocle. 
Let $\overline{W}$ and $\overline{W}^0$ be the reductions modulo $p$ of $W$ and $W^0$, respectively. Note that $K_1$ acts trivially on $\overline{W}$ and $\overline{W}^0$.

\begin{prop} \label{socle}
There are nonsplit exact sequences
\[ 0 \rarrow F(c+p-1,b,a-p+1) \rarrow \overline{W} \rarrow F(a-1,b,c+1) \rarrow 0\]
and
\[ 0 \rarrow F(a-1,b,c+1) \rarrow \overline{W}^0 \rarrow F(c+p-1,b,a-p+1) \rarrow 0\]
of $K$-representations.
\begin{proof}
The first exact sequence follows from \cite[Proposition 4.9]{herzigthesis}. The nonsplitness follows from \cite[Theorem 4.1]{herzigthesis}.
The second nonsplit exact sequence follows similarly after applying contragredients, where we use that the contragredient of $F(x,y,z)$ is $F(-z,-y,-x)$.
\end{proof}
\end{prop}

Fix an injection $i: W^0\inj W$ that is nonzero modulo $p$ (see \cite[II.8.13]{MR2015057} for particular choices) and let $i^\vee:W \inj W^0$ be the unique map that is nonzero modulo $p$ such that $i\circ i^\vee$ is a power of $p$.

\begin{prop} \label{alglat}
The composition $i\circ i^\vee$ is $p$.
\begin{proof}
This follows from \cite[II.8.15 (3)(4)]{MR2015057}.
We take $G$ to be $\GL_3$ so that $H^i$ in \emph{loc.~cit.}~is $\R^i \Ind_\BB^{\GL_3}$.
We take $\mu$ to be $w_0\nu$ so that $w_0 \cdot \mu = \nu - 2\rho$.
Let $\alpha_1$ and $\alpha_2$ be $(1,-1,0)$ and $(0,1,-1)$, respectively.
Then $s_{\alpha_i}$ in \emph{loc.~cit.}~is $s_i$.
Since $\tld{T}_\alpha(s_\alpha\cdot\mu) \circ \tld{T}_\alpha(\mu)$ is multiplication by $(\langle\mu,\alpha^\vee\rangle)!$ for all $\mu$, we see that $\tld{T}_{\alpha_1}(\mu)$ and $\tld{T}_{\alpha_1}(s_2s_1 \cdot \mu)$ are isomorphisms while $\tld{T}_{\alpha_2}(s_2 s_1\cdot \mu) \circ \tld{T}_{\alpha_2}(s_1\cdot \mu)$ is, up to a unit, multiplication by $p$.
As $\tld{T}_\alpha(\mu)$ is nonzero for all $\mu$, we can take $i = \tld{T}_{\alpha_1}(s_1 \cdot \mu) \circ \tld{T}_{\alpha_2}(s_2s_1 \cdot \mu) \circ \tld{T}_{\alpha_1}(w_0 \cdot \mu)$ and $i^\vee$ to be a unit multiple of $\tld{T}_{\alpha_1}(s_2s_1 \cdot \mu) \circ \tld{T}_{\alpha_2}(s_1\cdot \mu) \circ \tld{T}_{\alpha_1}(\mu)$.
This gives the desired result.
\end{proof}
\end{prop}

\subsection{Lattices in principal series types} \label{latps}

In this section, we classify lattices in generic principal series types which are residually multiplicity free.

\subsubsection{Principal series types in characteristic $p$}\label{subsubsec:mod p}

We describe the submodule structure of principal series types in characteristic $p$.
We begin with the submodule structure of inductions from the minimal parabolic subgroups
$P_1 = \Bigl\{\Bigl(\begin{smallmatrix}
*&*&*\\ *&*&* \\ 0&0&*
\end{smallmatrix} \Bigr)\Bigr\}, P_2 = \Bigl\{\Bigl(\begin{smallmatrix}
*&*&*\\ 0&*&* \\ 0&*&*
\end{smallmatrix} \Bigr)\Bigr\}$ of $\GL_3(\FF_p)$, strengthening \cite[Lemma 6.1.1]{MR3079258}.

For $s \in S_3$, let $s(a,b,c)$ be a dominant $p$-restricted character congruent mod $p-1$ to the $s$-permutation of $(a,b,c)$ (e.g.~if $(a,b,c)$ is $p$-restricted, then $r_1(a,b,c) = (c+p-1,a,b)$).

\begin{prop}\label{prop:parind}
Let $a,b,$ and $c$ be integers such that $0<a-b<p-1$ and $0<b-c<p-1$.
\begin{enumerate}
\item The socle and cosocle of $\Ind_{P_1}^{\GL_3(\FF_p)} F(a,b)\otimes F(c)$ are isomorphic to $F(r_1(a,b,c))$ and $F(a,b,c)$, respectively.
The socle and cosocle filtration have three nonzero graded pieces.
\item The socle and cosocle of $\Ind_{P_2}^{\GL_3(\FF_p)} F(a)\otimes F(b,c)$ are isomorphic to $F(r_2(a,b,c))$ and $F(a,b,c)$, respectively.
The socle and cosocle filtration have three nonzero graded pieces.
\end{enumerate}
\end{prop}
\begin{proof}
We prove only the first item as the proof of the second item is similar.
For $i = 1$ and $2$, let $N_i \subset P_i$ be the maximal normal $p$-subgroup, and $M_i = P_i/N_i$.
We begin by calculating the cosocle of $\Ind_{P_1}^{\GL_3(\FF_p)} F(a,b)\otimes F(c)$.
By Frobenius reciprocity, 
\[\Hom_{\GL_3(\FF_p)}(\Ind_{P_1}^{\GL_3(\FF_p)} F(a,b)\otimes F(c),F(x,y,z)) \cong \Hom_{M_1} (F(a,b)\otimes F(c),F(x,y,z)^{N_1}).\]
As $F(x,y)\otimes F(z) \subset F(x,y,z)^{N_1}$, and the latter is irreducible by \cite[Lemma 2.5(i)]{MR2845621}, we have that $F(x,y,z)^{N_1}\cong F(x,y)\otimes F(z)$.
Thus the $\Hom$-space above is nonzero if and only if it is one-dimensional and $F(x,y,z) \cong F(a,b,c)$.

The socle computation is similar.
By Frobenius reciprocity, 
\[\Hom_{\GL_3(\FF_p)}(F(x,y,z),\Ind_{P_1}^{\GL_3(\FF_p)} F(a,b)\otimes F(c)) \cong \Hom_{M_1} (F(x,y,z)_{N_1},F(a,b)\otimes F(c)).\]
By \cite[Lemma 2.5(ii)]{MR2845621}, $F(x,y,z)_{N_1} \cong F(x,y,z)^{\overline{N}_1}$.
Since $r_1 \overline{N}_1 r_1^{-1}=N_2$ and $F(x,y,z)^{N_2} \cong F(x)\otimes F(y,z)$, we see that $F(x,y,z)^{\overline{N}_1} \cong F(y,z)\otimes F(x)$.
Arguing as in the previous paragraph, we see that $(a,b,c) \equiv (y,z,x) \mod{(p-1)}$ and that $F(x,y,z) = F(c+p-1,a,b)$.

The remaining Jordan--H\"older factors of $\Ind_{P_1}^{\GL_3(\FF_p)} F(a,b)\otimes F(c)$ can be determined by the analogue of \cite[Lemma 6.1.1]{MR3079258}.
The crucial observation is that the remaining Jordan--H\"older factors do not admit extensions between them since they all lie in the same alcove.
Thus the socle and cosocle filtrations have exactly three nonzero graded pieces.
\end{proof}

Let $I \subset K$ (resp. $I_1\subset K$) be the inverse image of $\BB(\FF_p) \subset \GL_3(\FF_p)$ (resp. maximal $p$-subgroup of $\BB(\FF_p)$) under the reduction map $K \surj \GL_3(\FF_p)$.
Let $\etabar: \FF_p^\times \rarrow \FF_p^\times$ be the identity character.
For integers $a,b,$ and $c$, we consider the character $\chibar = \etabar^a \otimes \etabar^b \otimes \etabar^c$ of $(\FF_p^\times)^3$, which we view as a character of $I$ by inflation via the usual map $I \rarrow I/I_1 \isom (\FF_p^\times)^3$.
Let $\taubar$ be the $K$-representation $\ind_I^K \chibar$ over $\FF_p$.
We suppose that $a,b,$ and $c$ are integers such that $0<a-b<p-1$, $0<b-c<p-1$, and $a,b,$ and $c$ are distinct modulo $p-1$.
The following proposition describes the submodule structure of $\taubar$. Its proof was explained to us by Florian Herzig.

\begin{prop} \label{submodule} %\marginpar{Perhaps use $x,y,z$ instead of $a,b,c$}
If $a-c<p-1$ or $a-c>p-1$, then the socle and cosocle filtrations of $\taubar$ agree and have associated graded pieces 
\begin{center}
\[F(a,b,c)\] \\
\[F(a,c,b-p+1)\oplus F(b+p-1,a,c)\] \\
\[F(c+p-1,a,b)\oplus F(b-1,c,a-p+2) \oplus F(a-1,b,c+1)\]
\[\oplus F(c+p-2,a,b+1) \oplus F(b,c,a-p+1) \] \\
\[F(c+p-1,b,a-p+1),\]
\end{center}
or
\begin{center}
\[F(a,b,c)\] \\
\[F(a,c+p-1,b)\oplus F(b+p-2,a,c+p) \oplus F(c+p-2,b,a-p+2)\]
\[\oplus F(a-p,c,b-p+2) \oplus F(b,a-p+1,c)\] \\
\[F(c+2p-2,a,b) \oplus F(b,c,a-2p+2) \] \\
\[F(c+p-1,b,a-p+1),\]
\end{center}
respectively, where any number of bottom rows are the graded pieces of a submodule. Furthermore, all nontrivial extensions that can occur do occur (see \S \ref{notation}).
\begin{proof}
%A proof is sketched in \cite[\S 8]{MR0396731} in the case of $\SL_3(\FF_p)$.
%This proof works without modification for $\GL_3(\FF_p)$ as we now describe.
As $a,b,$ and $c$ are distinct mod $p-1$, the Jordan--H\"older factors of $\tau^e$ are distinct by \cite[Proposition 3.1, Theorems 4.1 and 5.1]{herzigthesis} and Proposition \ref{socle}.
Exactly as in the proof of Proposition \ref{prop:parind}, one can use Frobenius reciprocity and \cite[Lemma 2.5]{MR2845621}, to show that the socle and cosocle of $\taubar$ are $F(c+p-1,b,a-p+1)$ and $F(a,b,c)$, respectively.
Hence, $F(c+p-1,b,a-p+1)$ and $F(a,b,c)$ are pieces of the associated graded for both the socle and cosocle filtrations.

We now assume that $a-c<p-1$.
The other case can be treated similarly, or by using duality.
By \cite[Lemma 6.1.1]{MR3079258}, there is a natural inclusion $\Ind_{P_1}^{\GL_3(\FF_p)} F(b+p-1,a) \otimes F(c) \subset \taubar$.
From Proposition \ref{prop:parind}, we see that in the socle filtration, $F(b-1,c,a-p+2),F(a-1,b,c+1), F(c+p-2,a,b+1),$ and $F(b,c,a-p+1)$ are in the layer above the socle $F(c+p-1,b,a-p+1)$, and $F(b+p-1,a,c)$ is in the next layer up as taking socles respects inclusions.
Furthermore, all possible nontrivial extensions between these weights occur.
The same argument applies to $\Ind_{P_2}^{\GL_3(\FF_p)} F(a) \otimes F(c,b-p+1)$.
The only remaining weight is the cosocle $F(a,b,c)$, which must be in the final layer of the socle filtration.
Furthermore, since the cosocle is isomorphic to $F(a,b,c)$, this weight must extend $F(b+p-1,a,c)$ and $F(a,c,b-p+1)$.
Finally, we observe that every irreducible factor of $\taubar$, aside from $F(a,b,c)$, is extended nontrivially by at least one irreducible
factor in row of the socle filltration immediately above it, which implies that the cosocle filtration coincides with the socle filtration.
\end{proof}
\end{prop}

\subsubsection{Integral structure on principal series types} \label{intps}
We begin by defining natural integral structures on principal series.
Let $\eta: \FF_p^\times \rarrow \ZZ_p^\times$ be the Teichmuller character.
For integers $a,b,$ and $c$, we consider the character $\chi = \eta^a \otimes \eta^b \otimes \eta^c$ of $(\FF_p^\times)^3$, which we view as a character of $I$ by inflation via the map $I \rarrow I/I_1 \isom (\FF_p^\times)^3$.
Let $\tau^e$ be the $K$-representation $\ind_I^K \chi$ over $\ZZ_p$, and let $\tau = \tau^e \otimes_{\ZZ_p} \QQ_p$.
As in \S \ref{subsubsec:mod p}, suppose that $a,b,$ and $c$ are distinct modulo $p-1$, so that in particular $\tau$ is absolutely irreducible.
By reordering $a,b,$ and $c$ and adding multiples of $p-1$, let us now assume without loss of generality that $a>b>c$ and $a-c<p-1$.
Let $\chi^s$ be the character with factors permuted by $s \in S_3$.
Let $e$ be the identity, $w_0 = (13)$, $r_1 = (123)$, $r_2 = (132)$, $s_1 = (12)$, and $s_2 = (23)$.
Then for example $\chi^{r_1} = \eta^c \otimes \eta^a \otimes \eta^b$ and $\chi^{r_2} = \eta^b \otimes \eta^c \otimes \eta^a$.

For each $s\in S_3$, consider the $K$-representation $\tau^s = \ind_I^K \chi^s$ over $\ZZ_p$.
The representations $\tau^s$ are lattices in the principal series types $\tau^s \otimes_{\ZZ_p} \QQ_p \cong \tau$.
Then by Proposition \ref{submodule} and \cite[Lemma 4.1.1]{1305.1594}, for each $s\in S_3$, the lattice $\tau^s$ is the unique lattice up to homothety with cosocle $F(s(a,b,c))$.
Having described the reductions of lattices in $\tau$ with cosocle $F(s(a,b,c))$ in Proposition \ref{submodule}, we now describe the submodule structure of $\taubar^1$, $\taubar^2$, and $\taubar^3$ where $\tau^1$, $\tau^2,$ and $\tau^3 \subset \tau$  are the unique lattices up to homothety with cosocle $F(b-1,c,a-p+2), F(a-1,b,c+1),$ and $F(c+p-2,a,b+1)$, respectively.

\begin{prop} \label{submodule1'}
For $i = 1,2,$ or $3$, the cosocle filtration of $\taubar^i$ has three associated graded pieces.
The quotient is an irreducible lower alcove weight, the next layer is the direct sum of the upper alcove weights in $\mathrm{JH}(\taubar^i)$, and the final layer is the direct sum of the remaining lower alcove weights in $\mathrm{JH}(\taubar^i)$.
For example, for $\taubar^1$, the associated graded pieces are
\begin{center}
\[ F(b-1,c,a-p+2)\] \\
\[F(b+p-1,a,c) \oplus F(a,c,b-p+1) \oplus F(c+p-1,b,a-p+1)\] \\
\[ F(a,b,c) \oplus F(b,c,a-p+1) \oplus F(c+p-1,a,b)\]
\[\oplus F(c+p-2,a,b+1) \oplus F(a-1,b,c+1).\]
\end{center}
\begin{proof}
We will prove the proposition for $\taubar^1$, the other cases being similar.
The weight $F(b-1,c,a-p+2)$ is in the top layer for both filtrations by construction.
Recall from the proof of \cite[Lemma 4.1.1]{1305.1594}, $\tau^1 \subset \tau^{w_0}$ is the minimal submodule such that the quotient $\tau^{w_0}/\tau^1$ does not contain $F(b-1,c,a-p+2)$ as a Jordan--H\"older factor.
We claim that the image of $\taubar^1$ in $\taubar^{w_0}$ is the minimal submodule $M$ of $\taubar^{w_0}$ containing $F(b-1,c,a-p+2)$ as a Jordan--H\"older factor.
It clearly must contain it.
On the other hand, the quotient of $\tau^{w_0}$ by the preimage of $M$ does not contain $F(b-1,c,a-p+2)$ as a Jordan--H\"older factor, and so the preimage of $M$ must contain $\tau^1$ by minimality of $\tau^1$.

Proposition \ref{submodule} shows that $F(a,c,b-p+1)$ and $F(b+p-1,a,c)$ are in the second layer of the cosocle filtration of $\im \taubar^1$.
Since formation of the radical submodule commutes with quotients, $F(a,c,b-p+1)$ and $F(b+p-1,a,c)$ are in the second layer of the cosocle filtration of $\taubar^1$.
Similarly, looking at the inclusion $\tau^1 \subset \tau^{s_1}$, we see that $F(c+p-1,b,a-p+1)$ is also in the second layer of the cosocle filtration of $\taubar^1$.
In fact, $F(a,c,b-p+1) \oplus F(b+p-1,a,c) \oplus F(c+p-1,b,a-p+1)$ is the second layer of the cosocle filtration since these weights are the only Jordan--H\"older factors of $\taubar^1$ that nontrivially extend $F(b-1,c,a-p+2)$ (see \S \ref{notation}).
The other five weights are all in the third layer of the cosocle filtration since there are no nontrivial extensions between them (see \S \ref{notation}).
\end{proof}
\end{prop}

For $s,s' \in S_3$, the intertwiner $\iota_{s's}^s:\ind_{I}^{K} \chi^{s's} \inj \ind_{I}^{K} \chi^s$ from the classical representation theory of $\GL_3(\FF_p)$ is nonzero modulo $p$ and induces isomorphisms $\tau^{s's} \otimes_{\ZZ_p} \QQ_p \isom \tau^s \otimes_{\ZZ_p} \QQ_p$, as we now recall.

Let $v^s \in \ind_{I}^{K} \chi^s$ be the function supported on $I$ defined by extending $\chi^s$ by $0$.
By Frobenius reciprocity, the intertwiner $\iota_{s's}^s$ is determined by the image of $v^{s' s}$.
Identifying $S_3$ with the group of permutation matrices in $\GL_3$, the intertwiner is given by
\begin{equation} \label{inter}
\iota_{s's}^s: v^{s ' s} \mapsto \sum_{g\in I_1/(s' I_1 s'^{-1} \cap I_1)} gs' v^s,
\end{equation}
which is easily checked to be nonzero modulo $p$.
%Now suppose that $a>b>c$ with $a-c < p-2$.
%\begin{prop} \label{cosocle}
%For a pair $s,s'\in S_3$, the map $\iota_s^{s'}:\tau^s\rarrow \tau^{s'}$ is the unique up to scalar map which is nonzero modulo $p$.
%\begin{proof}
%This follows from the fact that $\Hom_{\ZZ_p[\GL_3(\FF_p)]}(\tau^s,\tau^{s'})$ is a lattice in 
%\[\Hom_{\QQ_p[\GL_3(\FF_p)]}(\tau^s \otimes_{\ZZ_p}\QQ_p,\tau^{s'}\otimes_{\ZZ_p}\QQ_p) \cong \QQ_p,\]
%and hence isomorphic to $\ZZ_p$.
%\end{proof}
%\end{prop}
Let $\ell$ denote the length function on $S_3$.

\begin{prop} \label{intercomp}
The composition $\iota_{s's}^{s''s's} \circ \iota_s^{s's}: \ind_{I}^{K} \chi^s \inj \ind_{I}^{K} \chi^{s''s's}$ is $p^{\frac{1}{2}(\ell(s'')+\ell(s') - \ell(s''s'))} \iota_s^{s''s's}$.
\begin{proof}
A more general result can be obtained from \cite[Propositions 3.6 and 3.10]{MR0396731}.
We provide a short summary.
Let the ordered pair $(i,j)$ be a permutation of $(1,2)$.
By induction, it suffices to show that $\iota_{s_i s}^{s_j s_i s} \circ \iota_s^{s_i s}=\iota_s^{s_js_is}$ and $\iota_{s_i s}^{s_i s_i s}\circ \iota_s^{s_i s} =p$.

Note that $\cup_{g\in I_1/(s' I_1 s'^{-1} \cap I_1)} gs'I_1 = I_1s'I_1$.
To prove that $\iota_{s_i s}^{s_j s_i s} \circ \iota_s^{s_i s}=\iota_s^{s_js_is}$, it suffices to show that the convolution of $\mathbf{1}_{I_1 s_j I_1}$ and $\mathbf{1}_{I_1 s_i I_1}$ is $\mathbf{1}_{I_1 s_j s_i I_1}$.
Since $I_1 s_j s_i I_1 \subset I_1 s_j I_1 \cdot I_1 s_i I_1$, we have that $\mathbf{1}_{I_1 s_j s_i I_1} \leq \mathbf{1}_{I_1 s_j I_1} \ast \mathbf{1}_{I_1 s_i I_1}$.
As $\#I_1 s_j s_i I_1/I_1 = \# I_1s_i I_1/I_1 \cdot \# I_1s_jI_1/I_1$, $\mathbf{1}_{I_1 s_j s_i I_1}$ and $\mathbf{1}_{I_1 s_j I_1} \ast \mathbf{1}_{I_1 s_i I_1}$ have the same integrals for any Haar measure.
We conclude that they are equal.
The identity $\iota_{s_i s}^{s_i s_i s}\circ \iota_s^{s_i s}=p$ reduces to a calculation for the $\GL_2$ minor coming from the $i$-th and $i+1$-th rows and columns.
This identity then follows from the computation in the proof of \cite[Lemme 2.2]{breuil}.
\end{proof}
\end{prop}

We now use Proposition \ref{submodule1} to complete the classification of maps between lattices with irreducible cosocle.
For $i = 1,2,3$, $s \in S_3$, let $\iota_i^s: \tau^i \rarrow \tau^s$ be fixed inclusions of lattices that are nonzero modulo $p$.
For $i$ and $j$ in the set $\{1,2,3\}$ with $i<j$, let $\iota_i^j: \tau^i \rarrow \tau^j$ be fixed inclusions of lattices that are nonzero modulo $p$.
These inclusions are unique up to unit scalar by \cite[Lemma 4.1.1]{1305.1594}.
Let $\iota_s^i: \tau^s \rarrow \tau^i$ and $\iota_j^i: \tau^j \rarrow \tau^i$ be the unique inclusions of lattices that are nonzero modulo $p$ such that $\iota_s^i \circ \iota_i^s$ and $\iota_j^i \circ \iota_i^j$ are powers of $p$.

\begin{prop} \label{prop:gauge}
The compositions of inclusions of lattices are given as follows:
 \begin{enumerate}
  \item $\iota_j^i \circ \iota_i^j = p^2$;
  \item if $s = s_1,s_2,$ or $w_0$, then $\iota_i^s \circ \iota_s^i = p$; and
  \item if $s = e,r_1,$ or $r_2$, then $\iota_i^s \circ \iota_s^i = p^2$.
 \end{enumerate}
\begin{proof}[Proof of $(2)$]
By relabelling $a,b,$ and $c$, we can assume without loss of generality that $i = 1$.
We use the symbol $\sim$ to denote equality up to a unit.
Let $s'$ be one of $s_1,s_2,$ or $w_0$ with $s'\neq s$.
By Proposition \ref{submodule1'}, the minimal submodule of $\taubar^1$ containing $F(s'(a,b,c))$ as a Jordan--H\"older factor does not contain $F(s(a,b,c))$ as a Jordan--H\"older factor.
Thus, the image of the composition $\iota_{s'}^1 \circ \iota_{s}^{s'}$ is zero modulo $p$, and hence a positive power of $p$ times $\iota_{s}^i$ up to a unit.
Then again using Proposition \ref{submodule}, we have $p^2 = \iota_{s'}^{s} \circ \iota_{s}^{s'} \sim \iota_1^{s} \circ \iota_{s'}^1 \circ \iota_{s}^{s'}$ so that $p^2$ is a positive power of $p$ times $\iota_1^{s}\circ \iota_{s}^1$.
We conclude that $\iota_1^{s}\circ \iota_{s}^1 = p$.
\end{proof}
\end{prop}

To prove (1) and (3), we need the following result.

\begin{prop} \label{submodule1}
For $i = 1,2,$ or $3$, the socle and cosocle filtrations of $\taubar^i$ agree and have three associated graded pieces $($see Proposition \ref{submodule1'}$)$.
Moreover, all nontrivial extensions between pieces that can occur do occur $($see \S \ref{notation}$)$.
\end{prop}
\begin{proof}
We first show the second statement.
Clearly all possible nontrivial extensions between the top two rows occur.
We now show that all possible nontrivial extensions between the bottom two rows occur.
Since the composition $\iota_1^{s_1}\circ \iota_{s_1}^1= p$ by Proposition \ref{prop:gauge}(2), the image of $\iota_1^{s_1}: \taubar^1 \rarrow \taubar^{s_1}$ and the cokernel of $\iota_{s_1}^1:\taubar^{s_1} \rarrow \taubar^1$ have the same Jordan--H\"older factors, namely $F(b-1,c,a-p+2)$, $F(a,c,b-p+1), F(c+p-1,b,a-p+1),$ and $F(c+p-1,a,b)$ by Proposition \ref{submodule}.
Hence all possible nontrivial extensions of $F(b+p-1,a,c)$ occur in $\taubar^1$ since they occur in $\im(\taubar^{s_1} \rarrow \taubar^1)$ by Proposition \ref{submodule}.
All other nontrivial extensions are established analogously.
The first statement follows from Proposition \ref{submodule1'} and the second statement.
\end{proof}

\begin{proof}[Proof of Proposition \ref{prop:gauge}$(1)$ and $(3)$]
We now turn to (1).
We continue to use the symbol $\sim$ to denote equality up to a unit.
By Propositions \ref{submodule} and \ref{submodule1}, we have $\iota_{s_1}^{w_0} \sim  \iota_i^{w_0} \circ \iota_{s_1}^i$, $\iota_{w_0}^{s_1} \sim \iota_j^{s_1} \circ \iota_{w_0}^j$, $\iota_j^i \sim \iota_{s_1}^i \circ \iota_j^{s_1}$, and $\iota_i^j \sim \iota_{w_0}^j \circ \iota_i^{w_0}$.
Hence $\iota_i^{w_0} \circ \iota_j^i \circ \iota_i^j \sim \iota_i^{w_0} \circ \iota_{s_1}^i \circ \iota_j^{s_1} \circ \iota_{w_0}^j \circ \iota_i^{w_0} \sim \iota_{s_1}^{w_0} \circ \iota_{w_0}^{s_1} \circ \iota_i^{w_0} \sim  p^2 \iota_i^{w_0}$ by Proposition \ref{intercomp}, and so $\iota_j^i \circ \iota_i^j = p^2$.

For (3), by relabelling $a,b,$ and $c$, we can assume without loss of generality that $s = e$.
From Propositions \ref{submodule} and \ref{submodule1}, we see that $\iota_e^i\sim\iota_{s_1}^i \circ \iota_e^{s_1}$ and $\iota_i^e \sim \iota_{s_1}^e \circ \iota_i^{s_1}$.
Hence $\iota_i^e \circ \iota_e^i \sim \iota_{s_1}^e \circ \iota_i^{s_1} \circ \iota_{s_1}^i \circ \iota_e^{s_1}=p^2$ by Proposition \ref{intercomp} and (2).
\end{proof}

\subsection{Lattices via Morita theory}\label{sec:morita}

In this subsection, we use Morita theory to explicitly describe the moduli of lattices in a given residually multiplicity
free $K$-representation.
The idea to use Morita theory to describe the moduli of lattices was suggested by David Helm (see \cite[\S 1.6]{1305.1594}).

\subsubsection{The abelian category generated by lattices}
Let $\tau$ be a continuous irreducible (finite dimensional) $K$-representation over $\QQ_p$, which is residually multiplicity free.
Let $\mathcal{L}$ be the category of all finitely generated $\ZZ_p$-modules with a $K$-action which are isomorphic to subquotients of $\ZZ_p$-lattices in $\tau^{\oplus n}$ for some $n\in \NN$.
The irreducible objects of $\mathcal{L}$ are $\sigma$ where $\sigma$ is a Jordan--H\"older factor of $\taubar$.
Let $\tau^\sigma \subset \tau$ be a lattice (unique up to homothety by \cite[Lemma 4.1.1]{1305.1594}) with cosocle $\sigma$.

Let $\Gamma$ be the group $\GL_3(\FF_p)$.
Let $P^\sigma\rarrow \sigma$ be a projective envelope in the category of $\ZZ_p[\Gamma]$-modules.
Then there exists a surjection $\varphi:P^\sigma \surj \tau^\sigma$, which we fix.
The following lemma appears in the proof of \cite[Proposition 4.2.1]{1305.1594}.

\begin{lem} \label{lemma:generator}
The space $\Hom_{\ZZ_p[\Gamma]}(P^\sigma,\tau^\sigma)$ is a free $\ZZ_p$-module generated by $\varphi$.
\begin{proof}
The space is clearly free over $\ZZ_p$, being finitely generated and torsion-free.
By projectivity of $P^\sigma$, we have
\[\Hom_{\ZZ_p[\Gamma]}(P^\sigma,\tau^\sigma)/p\Hom_{\ZZ_p[\Gamma]}(P^\sigma,\tau^\sigma) \cong \Hom_{\FF_p[\Gamma]}(\Pbar^\sigma,\taubar^\sigma) \cong \Hom_{\FF_p[\Gamma]}(\Pbar^\sigma,\taubar^{\mathrm{ss}}),\]
where the last isomorphism follows from projectivity of $\Pbar^\sigma$ as an $\FF_p[\Gamma]$-module.
Since $\Hom_{\FF_p[\Gamma]}(\Pbar^\sigma,\taubar^{\mathrm{ss}})$ is one-dimensional and generated by the image of $\varphi$, the result now follows from Nakayama's lemma.
\end{proof}
\end{lem}

\begin{lem}\label{lemma:quotient}
For any $M \in \mathcal{L}$, any $\ZZ_p[\Gamma]$-homomorphism $P^\sigma \rarrow M$ factors through $\varphi$.
\begin{proof}
Suppose that $M$ is a quotient of a lattice $N \subset \tau^{\oplus n}$.
Then the map $P^\sigma \rarrow M$ lifts to a map $P^\sigma \rarrow N$.
For some embedding $\tau^\sigma\subset \tau$, we have $N \subset (\tau^\sigma)^{\oplus n}$.
So we can assume without loss of generality that $M = (\tau^\sigma)^{\oplus n}$.
Since by Lemma \ref{lemma:generator} $\Hom_{\ZZ_p[\Gamma]}(P,(\tau^\sigma)^{\oplus n})$ is generated by maps which factor through $\varphi$, any such map factors through $\varphi$.
\end{proof}
\end{lem}

\begin{prop} \label{projcover}
The lattices $\tau^\sigma$ are projective objects in $\mathcal{L}$.
\begin{proof}
We show that $\tau^\sigma$ satisfies the lifting property.
Let $M\surj N$ be a surjection in $\mathcal{L}$.
We wish to lift a map $\tau^\sigma \rarrow N$ to a map $\tau^\sigma \rarrow M$.
The composition $P^\sigma \surj \tau^\sigma \rarrow N$ with $\varphi$ lifts to $P^\sigma \rarrow M$ by projectivity of $P^\sigma$, which factors through $\varphi$ by Lemma \ref{lemma:quotient}. 
Since $\varphi$ is surjective, this gives the required lifting.
\end{proof}
\end{prop}

Let $\mathcal{P} = \oplus_\sigma \tau^\sigma$ where $\sigma$ runs over the distinct Jordan--H\"older factors of $\taubar$.
Let $\mathcal{E} = \End_K(\mathcal{P})$.

\begin{prop} \label{prop:morita}
The functor $M \mapsto \Hom_K(\mathcal{P},M)$ gives an equivalence of categories $\mathcal{L} \rarrow \textrm{f.g. } \mathcal{E}^{\textrm{op}}$-modules.
\begin{proof}
By Proposition \ref{projcover}, $\mathcal{P}$ is projective generator of the category $\mathcal{L}$. The equivalence follows from Morita theory with quasi-inverse given by $(\cdot) \otimes_{\mathcal{E}^{\textrm{op}}} \mathcal{P}$.
\end{proof}
\end{prop}

\subsubsection{Moduli of lattices in algebraic vectors} \label{subsubsec:moduli}
We use the notation $V,W,$ and $W^0$ from \S \ref{algvect}.
Let $\mathcal{E}^\mathrm{alg}:= \End_K(W^0\oplus W)$.
Using Proposition \ref{alglat}, it is easy to calculate that there is an isomorphism
\[ \mathcal{E}^\mathrm{alg} \isom \left( \begin{array}{cc}
\ZZ_p & p\ZZ_p  \\
\ZZ_p & \ZZ_p  \end{array} \right),\]
where $\bigl(\begin{smallmatrix}
1&0\\ 0&0
\end{smallmatrix} \bigr)$ and
$\bigl(\begin{smallmatrix}
0&0\\ 0&1
\end{smallmatrix} \bigr)$
restrict to the identity on $W^0$ and $W$, respectively, and
$\bigl(\begin{smallmatrix}
0&p\\ 0&0
\end{smallmatrix} \bigr)$
and
$\bigl(\begin{smallmatrix}
0&0\\ 1&0
\end{smallmatrix} \bigr)$ correspond to $i^\vee$ and $i$, respectively.

We now use Proposition \ref{prop:morita} to describe the moduli space of lattices in $V$ (with rigidifications). Let $\cM$ be the functor which takes a $\ZZ_p$-algebra $R$ to the set
\begin{align*}
\{(\mathcal{P}_R,\psi_1,\psi_2)|&\mathcal{P}_R \textrm{ is a free } R\textrm{-module with an } \mathcal{E}^\mathrm{alg} \textrm{-action},\\
&\psi_1: R\isom\bigl(\begin{smallmatrix}
1&0\\ 0&0
\end{smallmatrix} \bigr)\mathcal{P}_R,
\psi_2:R \isom \bigl(\begin{smallmatrix}
0&0\\ 0&1
\end{smallmatrix} \bigr)\mathcal{P}_R\}/\cong
\end{align*}
We consider $\cE^\mathrm{alg}$-modules rather than $\cE^{\mathrm{alg},\textrm{op}}$-modules because of the variance of the patching construction in \S \ref{patching}.

\begin{prop} \label{rep}
The functor $\cM$ is represented by the ring $\ZZ_p[x,y]/(xy-p)$ where the map $\ZZ_p[x,y]/(xy-p) \rarrow R$ classifying $(\mathcal{P}_R,\psi_1,\psi_2)$ is given by $x \mapsto \psi_1^{-1} \circ \bigl(\begin{smallmatrix}
0&p\\ 0&0
\end{smallmatrix} \bigr) \circ \psi_2$ and $y \mapsto \psi_2^{-1} \circ \bigl(\begin{smallmatrix}
0&0\\ 1&0
\end{smallmatrix} \bigr) \circ \psi_1$.
\begin{proof}
Let $A = \ZZ_p[x,y]/(xy-p)$. Then we define $\mathcal{P}_A$ to be $A^2$ where $\mathcal{E}^\mathrm{alg} \rarrow M_2(A)$ is given by
\begin{align*}
\bigl(\begin{smallmatrix}
1&0\\ 0&0
\end{smallmatrix} \bigr) &\mapsto \bigl(\begin{smallmatrix}
1&0\\ 0&0
\end{smallmatrix} \bigr) \\
\bigl(\begin{smallmatrix}
0&0\\ 0&1
\end{smallmatrix} \bigr) &\mapsto \bigl(\begin{smallmatrix}
0&0\\ 0&1
\end{smallmatrix} \bigr) \\
\bigl(\begin{smallmatrix}
0&p\\ 0&0
\end{smallmatrix} \bigr) &\mapsto \bigl(\begin{smallmatrix}
0&x\\ 0&0
\end{smallmatrix} \bigr) \\
\bigl(\begin{smallmatrix}
0&0\\ 1&0
\end{smallmatrix} \bigr) &\mapsto \bigl(\begin{smallmatrix}
0&0\\ y&0
\end{smallmatrix} \bigr)
\end{align*}
The triple $(P_A,\id_A,\id_A)$ is universal since $\psi_1 \oplus \psi_2 \colon (R^2,\id_R,\id_R) \rarrow (\mathcal{P}_R,\psi_1,\psi_2)$ is an isomorphism. Here the $\mathcal{E}^\mathrm{alg}$-action of $R^2$ is given by the map $A \rarrow R$ described in the statement of the proposition.
\end{proof}
\end{prop}

\section{Hecke algebras} \label{hecke}

In this section, we describe integral Hecke operators for some locally algebraic types, and relate them to maps between lattices.
We give some general facts about Hecke algebras in \S \ref{heckered} that we specialize to extensions
of weights in \S \ref{heckealg} and to principal series types in \S \ref{heckeps}.

\subsection{Hecke algebras for split reductive groups} \label{heckered}

We first describe integral Hecke algebras for general representations of the $\ZZ_p$-points of split reductive groups.
Let $\GG$ be a split reductive group over $\ZZ_p$, and $\TT \subset \BB$ a maximal torus and a Borel subgroup.
Let $X_*(\TT)$ (resp. $X_*(\TT)_-$) denote the group of cocharacters (resp. antidominant cocharacters) of $\TT$. Let $K = \GG(\ZZ_p)$ be a maximal compact open hyperspecial subgroup of $G = \GG(\QQ_p)$ ~and
\[K_1 = \ker(\GG(\ZZ_p) \rarrow \GG(\FF_p)).\]
For $K$-representations $W_1$, and $W_2$ over $\ZZ_p$, define the Hecke bimodule
\[\mathcal{H}(W_1,W_2):= \Hom_G (\ind_K^G W_1,\ind_K^G W_2),\]
%\begin{align*}
%\mathcal{H}(W_1,W_2)&:= \Hom_G (\ind_K^G W_1,\ind_K^G W_2) \\
%&\cong \{\textrm{compactly supported }f :G\rarrow \Hom_{\ZZ_p}(W_1,W_2)\big{|} \\
%& \hspace{7mm} \forall k_1,k_2 \in K, g\in G, f(k_1gk_2) = k_1 \circ f(g) \circ k_2 \}.
%\end{align*}
which by Frobenius reciprocity is isomorphic to the space of compactly supported functions $f :G\rarrow \Hom_{\ZZ_p}(W_1,W_2)$ such that $f(k_2gk_1) = k_2 \circ f(g) \circ k_1$ for all $k_1,k_2 \in K$ and $g\in G$.
If $W$ is a $K$-representation, let $\mathcal{H}(W) = \mathcal{H}(W,W)$ be the Hecke algebra of $W$.
If $\pi$ is a $G$-representation over $\ZZ_p$, then $\Hom_K(W,\pi)$ is a right $\mathcal{H}_G(W)$-module.

Recall the Cartan decomposition
\[G = \sqcup_{\mu \in X_*(\TT)_-} K t_\mu K \]
where $t_\mu = \mu(p)$. This gives the decomposition
\begin{equation} \label{cartan}
\ind_{K}^G W = \bigoplus_{\mu \in X_*(\TT)_-} \Ind_K^{Kt_\mu K} W
\end{equation}
as $K$-representations
where
\[\Ind_K^{Kt K} W = \{f:KtK \rarrow W\big{|} f(kx) = kf(x) \hspace{2mm} \forall k\in K,x\in Kt K \} \]
and the $K$-action is the right regular action.
We can simplify this as follows:
\begin{equation} \label{herzighecke}
\Ind_K^{Kt K} W \cong \Ind_{K \cap tKt^{-1}}^{tK} W \cong \Ind_{t^{-1}Kt \cap K}^K W^{(t)}
\end{equation}
where $W^{(t)}$ is identified with $W$ as vector spaces, but the superscript denotes the twisted action on $W^{(t)}$ defined by $k w^{(t)} = (tkt^{-1} w)^{(t)}$ for $w^{(t)} \in W^{(t)}$ and $k \in t^{-1} K t$.

%\begin{align*}
%\Ind_K^{Kt K} W &\cong \{f:tK \rarrow W\big{|} f(tkk') = tkt^{-1}f(tk') \hspace{2mm} \forall k\in K\cap t^{-1}Kt,k'\in K\} \\
%&\cong \{f: K \rarrow W \big{|} f(kk') = tkt^{-1} f(k') \hspace{2mm} \forall k \in K\cap t^{-1}Kt,k'\in K \} \\
%&= \Ind_{K\cap t^{-1}Kt}^K W^{(t)}. \\
%\end{align*}

Given a coweight $\mu \in X_*(\TT)_-$, let $\mathcal{H}(W_1,W_2)_\mu = \Hom_K(W_1, \Ind_K^{Kt_\mu K} W_2) \subset \mathcal{H}(W_1,W_2)$ denote the subspace of elements supported on the double coset $Kt_\mu K$.
Let $\widetilde{N}_\mu$ and $\widetilde{P}_\mu$ be $t_\mu^{-1} K_1 t_\mu \cap K$ and $t_\mu^{-1} K t_\mu \cap K$, respectively.
The following proposition is contained in the proof of \cite[Theorem 1.2]{MR2771132}.

\begin{prop} \label{satake1}
Suppose that $K_1$ acts trivially on $W_1$ and $W_2$.
We have a natural isomorphism 
$\mathcal{H}(W_1,W_2)_\mu \cong \Hom_{\widetilde{P}_\mu} ((W_1)_{\widetilde{N}_\mu}, (W_2^{\widetilde{N}_{-\mu}})^{(t_\mu)})$.
\begin{proof}
We have a natural injection 
\begin{align*}
\Hom_{\widetilde{P}_\mu}((W_1)_{\widetilde{N}_\mu},(W_2^{\widetilde{N}_{-\mu}})^{(t_\mu)}) \inj \Hom_{\widetilde{P}_\mu}(W_1,W_2^{(t_\mu)}) &\cong \Hom_K(W_1,\Ind_{\widetilde{P}_\mu}^K W_2^{(t_\mu)}) \\
&\cong \mathcal{H}(W_1,W_2)_\mu
\end{align*}
where the first isomorphism follows from Frobenius reciprocity and the second isomorphism follows from (\ref{herzighecke}).
It suffices to show surjectivity of the first map. Let $f \in \Hom_{\widetilde{P}_\mu}(W_1,W_2^{(t_\mu)})$.
Since $W_1$ is $K_1$-invariant, the image of $f$ is contained in \[\(W_2^{(t_\mu)}\)^{K_1 \cap t_\mu^{-1} K t_\mu} = \(W_2^{t_\mu K_1 t_\mu^{-1}\cap K}\)^{(t_\mu)} = \(W_2^{\widetilde{N}_{-\mu}}\)^{(t_\mu)}.\]
Similarly, since $W_2$ is $K_1$-invariant, $W_2^{(t_\mu)}$ is $\widetilde{N}_\mu$-invariant.
Therefore the map $f$ factors through $(W_1)_{\widetilde{N}_\mu}$.
\end{proof}
\end{prop}

We can rephrase Proposition \ref{satake1} as follows.
Given a $\widetilde{P}_\mu$-morphism $(W_1)_{\widetilde{N}_\mu} \rarrow (W_2^{\widetilde{N}_{-\mu}})^{(t_\mu)}$, we obtain a $K$-morphism $\Ind_{\widetilde{P}_\mu}^K (W_1)_{\widetilde{N}_\mu} \rarrow \Ind_{\widetilde{P}_\mu}^K (W_2^{\widetilde{N}_{-\mu}})^{(t_\mu)}$.
By Frobenius reciprocity, we have a natural map $W_1 \rarrow \Ind_{\widetilde{P}_\mu}^K (W_1)_{\widetilde{N}_\mu}$. 
%and $\Ind_{\widetilde{P}_\mu}^K(W_2^{\widetilde{N}_{-\mu}})^{(t_\mu)} \rarrow W_2^{(t_\mu)}$.
The composition
\[ W_1 \rarrow \Ind_{\widetilde{P}_\mu}^K (W_1)_{\widetilde{N}_\mu} \rarrow \Ind_{\widetilde{P}_\mu}^K (W_2^{\widetilde{N}_{-\mu}})^{(t_\mu)}
%\rarrow W_2^{(t_\mu)} 
\subset \Ind_{\widetilde{P}_\mu}^K W_2^{(t_\mu)} \cong \Ind_K^{Kt_\mu K} W_2\]
is the corresponding element of $\cH(W_1, W_2)_\mu$.

We can further simplify this when $\mu$ is minuscule, which we now suppose.
See \S 3 and particularly \cite[Proposition 3.8]{MR2771132} for a more general context.
Let $N_{\mu}$ and $P_{\mu}$ be the images of $\widetilde{N}_{\mu}$ and $\widetilde{P}_{\mu}$ in $\GG(\FF_p)$, respectively.
Then $N_{\mu}$ and $P_{\mu}$ are the usual unipotent and parabolic subgroups in $\GG(\FF_p)$, respectively, corresponding to $\mu$.
Let $M_\mu = P_\mu/N_\mu$.
Note that $t_\mu \widetilde{P}_\mu t_\mu^{-1} = \widetilde{P}_{-\mu}$ and $t_\mu \widetilde{N}_\mu K_1 t_\mu^{-1} = \widetilde{N}_{-\mu} K_1$, and so conjugation by $t_\mu$ gives an isomorphism $M_\mu \cong M_{-\mu}$.

\begin{prop} \label{satake}
Suppose that $\mu$ is minuscule and that $K_1$ acts trivially on $W_1$ and $W_2$.
We have a natural isomorphism 
$\mathcal{H}(W_1,W_2)_\mu \cong \Hom_{M_\mu} ((W_1)_{N_\mu}, W_2^{N_{-\mu}})$
where $M_\mu$ acts on $W_2^{N_{-\mu}}$ through the isomorphism $M_\mu \isom M_{-\mu}$ given by conjugation by $t_\mu$.
\begin{proof}
Since $\mu$ is minuscule, $K_1 \subset \widetilde{P}_\mu$. Using Proposition \ref{satake1} we have
\begin{align*}
\mathcal{H}(W_1,W_2)_\mu &\cong
\Hom_{\widetilde{P}_\mu} ((W_1)_{\widetilde{N}_\mu}, (W_2^{\widetilde{N}_{-\mu}})^{(t_\mu)})\\
&\cong \Hom_{\widetilde{P}_\mu} ((W_1)_{\widetilde{N}_\mu K_1}, (W_2^{\widetilde{N}_{-\mu}K_1})^{(t_\mu)}) \\
&\cong \Hom_{M_\mu} ((W_1)_{N_\mu}, W_2^{N_{-\mu}}).
\end{align*}
\end{proof}
\end{prop}

We specialize the above discussion to the case when $W_1 = W_2 = \sigma$ is an irreducible $\GG(\FF_p)$-representation over $\FF_p$ (and hence a $K$-representation by inflation) and $\mu$ is a minuscule coweight.
By \cite[Lemma 2.5]{MR2845621}, $\sigma_{N_\mu}$ is irreducible, and the composition of canonical maps $\sigma^{N_{-\mu}} \subset \sigma \surj \sigma_{N_\mu}$ is an isomorphism of $M_\mu\cong M_{-\mu}$-representations.
Recalling that the isomorphism $M_\mu\cong M_{-\mu}$ is given by $t_\mu$-conjugation, we get a $\widetilde{P}_\mu$-isomorphism $\sigma_{N_\mu} \isom t_{-\mu} \cdot \sigma^{N_{-\mu}} \subset \ind_{\widetilde{P}_{-\mu}}^G \sigma^{N_{-\mu}}$.
This composite induces a nonzero map 
\begin{equation}\label{def:T}
T'_{\sigma,\mu}: \ind_{\widetilde{P}_\mu}^G \sigma_{N_\mu} \rarrow \ind_{\widetilde{P}_{-\mu}}^G \sigma^{N_{-\mu}}.
\end{equation}
As $\cH(\sigma)_\mu$ is one-dimensional by {\it loc. cit.}~and Proposition \ref{satake}, the composition of $T'_{\sigma,\mu}$ with the natural inclusion $\ind_{\widetilde{P}_{-\mu}}^G \sigma^{N_{-\mu}} \cong \ind_{\widetilde{P}_{\mu}}^G (\sigma^{N_{-\mu}})^{(t_\mu)} \subset \ind_K^{Kt_\mu K} \sigma$ is a generator of $\cH(\sigma)_\mu$.
The following is a rephrasing of weight cycling as in \cite[Proposition 2.3.1]{MR3079258}.

\begin{prop} \label{weightcycling}
The map $($\ref{def:T}$)$ is an isomorphism.
\begin{proof}
Symmetrically to the definition of $T'_{\sigma,\mu}$, one can define $\widetilde{P}_{-\mu}$-isomorphism $\sigma^{N_{-\mu}} \isom t_{\mu} \cdot \sigma_{N_{\mu}} \subset \ind_{\widetilde{P}_{\mu}}^G \sigma_{N_{\mu}}$.
The composite induces a map $\ind_{\widetilde{P}_{-\mu}}^G \sigma^{N_{-\mu}} \rarrow \ind_{\widetilde{P}_\mu}^G \sigma_{N_\mu}$ which is the inverse of $T'_{\sigma,\mu}$.
\end{proof}
\end{prop}

We omit $\mu$ from the notation if it is clear from context.
The following proposition follows from the discussion after Proposition \ref{satake1}.

\begin{prop} \label{extensioncycling}
If $T' \in \cH(W_1,W_2)_\mu$ corresponds via Proposition \ref{satake} to a map that factors through $\sigma_{N_\mu} \cong \sigma^{N_{-\mu}}$, then $T'$ factors through $T'_{\sigma,\mu}$ as $\ind_K^G W_1 \rarrow \ind_K^G \Ind_{\widetilde{P}_\mu}^K \sigma_{N_\mu} \rarrow \ind_K^G \Ind_{\widetilde{P}_{-\mu}}^K \sigma^{N_{-\mu}} \rarrow \ind_K^G W_2$.
\end{prop}

\subsection{Hecke algebras of extensions for $\GL_3$} \label{heckealg}

We specialize \S \ref{heckered} to the case of $\GG = \GL_3$ and certain extensions of weights, extending some of the results in \cite{MR2845621}.
Let $\mu$ be an antidominant minuscule coweight.
If $\mu$ is $(0,0,1)$ (resp. $(0,1,1)$), then $N_\mu = \overline{N}_1$ (resp. $\overline{N}_2$) in the notation of \S \ref{subsubsec:mod p}.
Let $a,b,$ and $c$ be integers such that $a>b>c$ and $a-c<p-1$.
Let $\sigma'$ be the upper alcove Serre weight $F(c+p-1,b,a-p+1)$ so that $\sigma'_{N_\mu} \cong F(b+p-1,a)\otimes F(c)$ or $F(a)\otimes F(c,b-p+1)$.
Let $\sigma$ be $F(a-1,b,c+1)$.
Recall that we defined $\overline{W}$ and $\overline{W}^0$ in \S \ref{algvect}.
Then $\overline{W}$ is a nonsplit extension of $\sigma$ by $\sigma'$ and $\overline{W}^0$ is a nonsplit extension of $\sigma'$ by $\sigma$.
Note that there is an injection $\overline{W}\inj\Ind_{\widetilde{P}_\mu}^K \sigma'_{N_\mu} = \Ind_{P_{\mu}}^{\GL_3(\FF_p)} \sigma'_{N_\mu}$.

\begin{prop} \label{modphecke}
We have $\overline{W}_{N_\mu} \cong \sigma_{N_\mu} \oplus \sigma'_{N_\mu}, \overline{W}^{N_{-\mu}} \cong (\sigma')^{N_{-\mu}}, \overline{W}^0_{N_\mu} \cong \sigma'_{N_\mu}$, and $(\overline{W}^0)^{N_{-\mu}} \cong \sigma^{N_{-\mu}} \oplus (\sigma')^{N_{-\mu}}$.
The spaces $\mathcal{H}(\overline{W})_\mu$ and $\mathcal{H}(\overline{W}^0)_\mu$ are both one-dimensional over $\FF_p$.
\begin{proof}
We have exact sequences
\begin{align*}
0 \rarrow (\sigma')^{N_{-\mu}} \rarrow &\overline{W}^{N_{-\mu}} \rarrow \sigma^{N_{-\mu}} \\
0 \rarrow \sigma^{N_{-\mu}} \rarrow &(\overline{W}^0)^{N_{-\mu}} \rarrow (\sigma')^{N_{-\mu}} \\
\sigma'_{N_\mu} \rarrow &\overline{W}_{N_\mu} \rarrow \sigma_{N_\mu} \rarrow 0 \\
\sigma_{N_\mu} \rarrow &\overline{W}^0_{N_\mu} \rarrow \sigma'_{N_\mu} \rarrow 0.
\end{align*}
By \cite[Lemma 2.5]{MR2845621} and the fact that $c \not\equiv a-1 \pmod{p-1}$, the irreducible $M_\mu$-representations $(\sigma')^{N_{-\mu}} \cong \sigma'_{N_\mu}$ and $\sigma^{N_{-\mu}} \cong \sigma_{N_\mu}$ have different actions of the $\GL_1$ factor of $M_\mu$.
Thus, if the above exact sequences extend to short exact sequences, they must split.
It suffices then to determine whether these exact sequences extend to short exact sequences.

Since $\sigma'$ is not a Jordan--H\"older factor of $\Ind_{P_{-\mu}}^{\GL_3(\FF_p)} \sigma^{N_{-\mu}}$ (resp. $\Ind_{{P}_{\mu}}^{\GL_3(\FF_p)} \sigma_{N_{\mu}}$) as can be checked casewise using \cite[Lemma 6.1.1]{MR3079258}, any map $\Ind_{P_{-\mu}}^{\GL_3(\FF_p)} \sigma^{N_{-\mu}} \rarrow \overline{W}$ (resp. $\overline{W}^0 \rarrow \Ind_{P_\mu}^{\GL_3(\FF_p)} \sigma_{N_\mu}$) must be zero.
By Frobenius reciprocity, we conclude that $\overline{W}^{N_{-\mu}} \cong (\sigma')^{N_{-\mu}}$ and $\overline{W}^0_{N_\mu} \cong \sigma'_{N_\mu}$.

On the other hand, we have an inclusion $\overline{W}\subset \Ind_{P_{\mu}}^{\GL_3(\FF_p)} \sigma'_{N_{\mu}}$.
By Frobenius reciprocity, we conclude that $\overline{W}_{N_\mu} \cong \sigma_{N_\mu} \oplus \sigma'_{N_\mu}$.
We claim that there exists a surjection $\Ind_{P_{-\mu}}^{\GL_3(\FF_p)} (\sigma')^{N_{-\mu}} \surj \overline{W}^0$, as can be checked from Proposition \ref{prop:parind}.
By Frobenius reciprocity, we conclude that $(\overline{W}^0)^{N_{-\mu}} \cong \sigma^{N_{-\mu}} \oplus (\sigma')^{N_{-\mu}}$.
The second part now follows from Proposition \ref{satake}.
\end{proof}
\end{prop}

\begin{remark}
An analogous result (with some hypotheses on $a$,$b$, and $c$) could be proven for $\sigma$ one of the weights $F(c+p-2,a,b+1)$ and $F(b-1,c,a-p+2)$ and $\overline{W}$ replaced by the extension of $\sigma$ by $\sigma'$ and $\overline{W}^0$ similarly replaced, but we will not use this in what follows.
\end{remark}

%Fix maps $W_2 \surj \sigma'$ and $\sigma' \inj W_1$ and denote the composition by $i$.
%The composition $(W_2)_{N_\mu} \surj \sigma'_{N_\mu} \isom \sigma'^{N_{-\mu}} \inj W_1^{N_{-\mu}}$, where first and third maps are induced by the above fixed maps and the second map is the inverse of the natural isomorphism, gives an element $T'\in \cH(W_1,W_2)_\mu$ as in Proposition \ref{extensioncycling}.
%We denote the images of $T'$ under the maps $\cH(W_1,W_2) \rarrow \cH(W_1)$ and $\cH(W_1,W_2) \rarrow \cH(W_2)$, induced by $i$, by $T_{W_1}$ and $T_{W_2}$, respectively.

%Now let $W_1$ and $W_2$ be $\overline{W}$ and $\overline{W}^0$, respectively, where $W = W(\nu)$ and $W^0$ are as in \S \ref{algvect}.
%Recall that $V$ is $W\otimes_{\ZZ_p} \QQ_p$.
%There is an injection $\overline{W} \inj \Ind_{P_{\mu}}^{\GL_3(\FF_p)} \sigma'_{N_{\mu}}$ for either antidominant minuscule $\mu$ by Proposition \ref{prop:parind}.
%Let $i: W^0 \rarrow W$ be the inclusion from \S \ref{algvect}.
%Let $T' \in \cH(\overline{W},\overline{W}^0)_\mu$ as above.

\begin{prop} \label{rank}
The standard $G$-action on $V$ gives an isomorphism $\ind_K^G V \cong V\otimes \ind_K^G \mathbf{1}$ which induces a canonical isomorphism $\mathcal{H}(V)_\mu \cong \mathcal{H}(\mathbf{1})_\mu$.
For an arbitrary $\mu \in X_*(\TT)_-$, $\cH(W, W_0)_\mu$, $\mathcal{H}(W^0)_\mu$, and $\mathcal{H}(W)_\mu$ are canonically lattices in the one dimensional space $\mathcal{H}(V)_\mu \cong \mathcal{H}(\mathbf{1})_\mu$ via $i: W_0\inj W \subset V$.
Moreover, the double coset operator $p^{\langle \mu,w_0\nu\rangle} Kt_\mu K$ lies in all these lattices.
\begin{proof}
The isomorphism $\mathcal{H}(V)_\mu \cong \mathcal{H}(\mathbf{1})_\mu$ comes from the fact that $V$ is irreducible.
The space $\mathcal{H}(\mathbf{1})_\mu$ is one-dimensional over $\QQ_p$ by the Satake isomorphism.
Since $t_\mu$ acts on the $\lambda$-weight space of $V^*$ by $p^{\langle \mu,\lambda\rangle}$ and $-w_0\nu$ is the highest weight of $V^*$, $p^{\langle \mu,w_0\nu\rangle} Kt_\mu K$ stabilizes 
\[\mathcal{H}(W) \cong \Hom_K(W,\ind_K^G W) \cong (W^d \otimes \ind_K^G W)^K \subset V^* \otimes \ind_K^G V\]
and thus is the image of an element in $\mathcal{H}(W)_\mu$.
(Recall that $W^d$ is the Schikhof or $\ZZ_p$-dual of $W$.)
The proof for the other lattices are the same.
\end{proof}
\end{prop}

The reduction $i: \overline{W}^0 \rarrow \overline{W}$ factors through a map $\sigma' \inj \overline{W}$ which we fix.
Let $T' \in \cH({W},{W}^0)_\mu$ be the element which maps to the double coset operator $p^{\langle \mu,w_0\nu\rangle} Kt_\mu K$ in Proposition \ref{rank}.

\begin{prop} \label{iso}
The inclusion $i$ induces natural maps $\cH(W, W^0)_\mu \rarrow \cH(W^0)_\mu$ and $\cH(W, W^0)_\mu \rarrow \cH(W)_\mu$, which are isomorphisms between rank one $\ZZ_p$-modules.
Furthermore, the map $\cH(\overline{W}, \overline{W}^0)_\mu \rarrow \cH(\sigma')_\mu$, induced by the fixed maps $\sigma' \inj \overline{W}$ and $i$, is an isomorphism.
Finally, $T'$ generates $\cH({W},{W}^0)_\mu$ over $\ZZ_p$.
\begin{proof}
As in the proof of Proposition \ref{rank}, $p^{\langle \mu,w_0\nu\rangle} t_\mu$ stabilizes on $W^d$.
Mod $p$, it acts by zero on all weight spaces except for the $-w_0\nu$-weight space.
Let $h$ be the natural isomorphism from $\mathcal{H}(\sigma')$ to the space of compactly supported functions $f: G \rarrow \End_K(\sigma')$ satisfying $f(k_1gk_2) = k_1 \circ f(g) \circ k_2$.
If $T$ is the image of $T'$ in $\mathcal{H}(W)$, then the reduction of $T$ mod $p$ stabilizes $\ind_K^G \sigma'$ and $h(T)(p)$ is the projection to the $w_0\nu$-weight space.
In particular, the map $\cH(\overline{W}, \overline{W}^0)_\mu \rarrow \cH(\sigma')_\mu$ is nonzero and thus an isomorphism since the domain and codomain are both one-dimensional by Proposition \ref{modphecke}.
Furthermore, the maps $\cH(\overline{W}, \overline{W}^0)_\mu \rarrow \cH(\overline{W})_\mu$ and $\cH(\overline{W}, \overline{W}^0)_\mu \rarrow \cH(\overline{W}^0)_\mu$ induced by $i$ are isomorphisms.
By Nakayama's lemma, the maps $\cH({W}, {W}^0)_\mu \rarrow \cH({W})_\mu$ and $\cH({W}, {W}^0)_\mu \rarrow \cH({W}^0)_\mu$ induced by $i$ are isomorphisms and $T'$ generates $\cH({W},{W}^0)_\mu$ over $\ZZ_p$.
\end{proof}
\end{prop}

By abuse of notation, let $T$ denote the generator which is the image of $T'$ in $\cH(W)_\mu,\cH(\overline{W})_\mu, \cH(W^0)_\mu, \cH(\overline{W}^0)_\mu,$ and $\cH(\sigma')_\mu$ via the maps in Proposition \ref{iso}.
We record this in the following proposition.

\begin{prop} \label{factor}
The morphism $T \in \mathcal{H}(W^0)_\mu$ factors as $\ind_K^G W^0 \overset{i}{\rarrow} \ind_K^G W \overset{T'}{\rarrow} \ind_K^G W^0$.
\end{prop}

\subsection{Hecke algebras for principal series types} \label{heckeps}

In this section, we describe Hecke operators for principal series types.
We use the notation of \S \ref{intps}, with the following abuses of notation: $v^s$ also denotes the image of $v^s \in \ind^K_I \chi^s$ under the natural map $\ind^K_I \chi^s \rarrow \ind^G_I\chi^s$, and $\iota_{s'}^s$ also denotes $\ind^G_K (\iota_{s'}^s)$.
%Assume moreover that $a-1>b>c+1$ and $a-c<p-1$.
%Note that by Frobenius reciprocity, we have 
%\[\Hom_{I}(\chi, \ind_{I}^{K} \chi) \cong \Hom_K(\ind_{I}^{K} \chi,\ind_{I}^{K} \chi) \cong \ZZ_p\]
%by irreducibility of $\ind_{I}^{K} \chi$.
%Hence, for each $s \in S_3$, $v^s \in \tau^s$ generates the $\chi^s$-isotypic submodule of $\tau^s$.
%    and write $v$ for $v^e$.
Let $t_1 = \Bigl(\begin{smallmatrix}
p^{-1}&0&0\\ 0&1&0 \\ 0&0&1
\end{smallmatrix} \Bigr)$ and $t_2 = \Bigl(\begin{smallmatrix}
p^{-1}&0&0\\ 0&p^{-1}&0 \\ 0&0&1
\end{smallmatrix} \Bigr)$.
Let $(i,j)$ be a permutation of $(1,2)$.

\begin{prop} \label{normalizer}
The element $r_i t_i v^s$ spans an $I_1$-invariant subspace of $\ind_{I}^G \chi^s$ on which $I$ acts by $\chi^{r_i s}$.
Moreover, the map $v^{r_i s} \mapsto t_i r_i v^s$ induces via Frobenius reciprocity isomorphisms $\nu_{r_i s}^s: \ind_{I}^G \chi^{r_i s} \rarrow \ind_{I}^G \chi^s$.
\begin{proof}
%The proof is similar to that of Proposition \ref{weightcycling}.
First, since $(t_ir_i)^{-1}I_1t_i r_i$ is $I_1$, $t_i r_i v^s$ spans an $I_1$-invariant subspace.
Second, we have that $t t_i r_i v^s = t_i r_i (r_i^{-1}tr_i) v^s = \chi^{r_i s}(t) t_i r_i v^s$ for any $t\in \TT(\ZZ_p)$.

For the second statement, note that the composition 
\[\nu_{r_i s}^s \circ \nu_s^{r_i s}: \ind_{I}^G \chi^s \rarrow \ind_{I}^G \chi^{r_i s} \rarrow \ind_{I}^G \chi^s,\] which maps $v^s$ to $r_i t_i r_j t_j v^s = \Bigl(\begin{smallmatrix}
p&0&0\\ 0&p&0 \\ 0&0&p
\end{smallmatrix} \Bigr)v^s$, is invertible and hence an isomorphism.
\end{proof}
\end{prop}

Let $U_i$ be the operator on $\ind_{I}^G \chi^s$ given by $v^s \mapsto \sum_{g\in t_i^{-1}I_1t_i/(t_i^{-1} I_1 t_i \cap I_1)} t_i g v^s$.

\begin{prop} \label{uop}
The composition 
\[\iota_{r_i s}^s\circ \nu_s^{r_i s}:\ind_{I}^G \chi^s \isom \ind_{I}^G \chi^{r_is} \inj \ind_{I}^G \chi^s\]
is the operator $U_j$.
\begin{proof}
From the respective formulas in (\ref{inter}) and Proposition \ref{normalizer}, we see that the composition is defined by 
\begin{align*}
v^{r_i s} \mapsto \sum_{g\in I_1/(r_i I_1 r_i^{-1} \cap I_1)} (t_j r_j) (gr_i)  v^{r_i s} &= \sum_{g\in r_i^{-1}I_1r_i/(r_i^{-1} I_1 r_i \cap I_1)} t_j g v^s\\
&=\sum_{g\in t_j^{-1}I_1t_j/(t_j^{-1} I_1 t_j \cap I_1)} t_j g v^s.
\end{align*}
\end{proof}
\end{prop}

\section{Patching} \label{patching}

Following \cite{1305.1594}, we use Kisin's local-global modification of the Taylor--Wiles method to study lattices in completed cohomology in families over local deformation rings.
In this section, we introduce the Taylor--Wiles patching method, as in \cite{1310.0831}, in order to deduce Theorem \ref{patchedmod}.
This result crucially shows that patched modules have a natural action by integral Hecke algebras at $p$.

\subsection{Galois representations and automorphic forms}
We now consider some Galois representations and automorphic forms relevant to our implementation of the Taylor--Wiles patching method.
Let $F$ be a CM field with $F^+$ its totally real subfield such that $F/F^+$ is unramified at all finite places and places of $F^+$ dividing $p$ split in $F$.

\subsubsection{Galois representations} \label{galrep}
Following \cite[\S 2]{MR2470687}, let $\cG/\ZZ$ be the group scheme defined to be the semidirect product of $\GL_3 \times \GL_1$ by the group $\{1,j\}$, which acts on $\GL_3\times \GL_1$ by $j(g,s) = (s\cdot (g^{-1})^t,s)$.
Let $E$ be a finite extension of $\QQ_p$ with ring of integers $\cO_E$, uniformizer $\varpi_E$, and residue field $k_E$.
We consider a continuous Galois representation $\rhobar:G_{F^+} \rarrow \cG(k_E)$ such that $\rbar^{-1} (\GL_3(k_E) \times \GL_1(k_E)) = G_F$.
Then the composition $\rbar:G_F \rarrow \GL_3(k_E) \times \GL_1(k_E) \rarrow \GL_3(k_E)$ of $\rhobar|_{G_F}$ and the first projection is an essentially conjugate self-dual Galois representation.
Recall that $\epsilon$ denotes the cyclotomic character and $\overline{\epsilon}$ its reduction.
Assume that $\rbar^c \cong \rbar^\vee \overline{\epsilon}^{-2}$ where $\cdot^c$ denotes complex conjugation.
Suppose that $\rbar$ ramifies only at places that are split in $F/F^+$. % and that $r|_{F_w}$ is potentially crystalline at places $w|p$ of $F$.
%As usual, we let $\rbar:G_F \rarrow \GL_3(k_E)$ denote the reduction of $r$.
\begin{ax} \label{tw}
We say that $\rbar$ satisfies the Taylor--Wiles conditions {\bf (TW)} if
\begin{itemize}
\item $\rbar$ has image containing $\GL_3(\FF)$ for some $\FF \subset k_E$ with $\#\FF > 9$, and
\item $\overline{F}^{\ker \ad \rbar}$ does not contain $F(\zeta_p)$.
\end{itemize}
\end{ax}

Note that the first condition, which is stronger than the usual condition of adequacy (see \cite[Definition 2.3]{MR2979825}), allows us to choose a place $v_1$ of $F^+$ which is prime to $p$ and satisfies the following properties (see \cite[\S 2.3]{1310.0831}). %In fact, we only need that $\#\FF \geq 7$.
\begin{itemize}
 \item $v_1$ splits in $F$ as $v_1 = w_1 w_1^c$
 \item $v_1$ does not split completely in $F(\zeta_p)$. 
 \item $\rbar(\Frob_{w_1})$ has distinct $\FF$-rational eigenvalues, no two of which have ratio $(\N v_1)^{\pm 1}$.
\end{itemize}
See Remark \ref{v1} for an explanation of the choice of $v_1$.

Let $\Sigma$ be a finite set of places of $F^+$ away from $p$ such that $\rbar$ is unramified away from places dividing $p$ and those in $\Sigma$.
Let $S$ be the set of places $v_1$, those in $\Sigma$, and those dividing $p$ in $F^+$.
For each place $v$ in $S$, fix a place $\widetilde{v}$ in $F$ lying over $v$
% with a fixed isomorphism $G_{F^+_v} \cong G_{F_{\widetilde{v}}}$.
and let $\widetilde{S}$ be the set of these places $\widetilde{v}$.
For places $w$ of $F$, let $R_w^\square$ be the universal $\mathcal{O}_E$-lifting ring of $\rbar|_{G_{F_w}}$.
For $v\in \Sigma$, let $R_{\widetilde{v}}^{\square,\tau_{\tld{v}}}$ be the reduced $p$-torsion free quotient of $R_{\widetilde{v}}^\square$ corresponding to lifts of the inertial type $\tau_{\tld{v}}$. % of $r|_{G_{F_{\widetilde{v}}}}$.
%In the remainder of this section, we forget the representation $r$, and use $\rbar$, $\Sigma$, $\widetilde{v}_1$, and the inertial types $\tau_{\tld{v}}$ as the inputs into the patching construction.

\subsubsection{The ring $R_\infty$}
Let $\delta_{F/F^+}$ denote the quadratic character of $F/F^+$.
Then we define
\[\cS = \(F/F^+,S,\widetilde{S},\cO_E,\rbar, \epsilon^{-2} \delta_{F/F^+}, \{ R^\square_{\widetilde{v}_1} \} \cup \{ R^{\square}_{\widetilde{v}}\}_{v|p} \cup \{ R^{\square,\tau_{\tld{v}}}_{\widetilde{v}} \}_{v \in \Sigma}\)\]
to be the deformation problem in the terminology of \cite{MR2470687}.
There is a universal deformation ring $R_\cS^{\textrm{univ}}$ and a universal $S$-framed deformation ring $R_\cS^{\square_S}$ in the sense of \cite[Definition 1.2.1]{MR2470687}.

Let 
\[R^{\textrm{loc}} = \widehat{\otimes}_{v|p} R^\square_{\widetilde{v}} \widehat{\otimes} \(\widehat{\otimes}_{v\in \Sigma} R^{\square,\tau_{\tld{v}}}_{\widetilde{v}}\) \widehat{\otimes} R^\square_{\widetilde{v}_1}\]
where all completed tensor products are taken over $\cO_E$.
Choose an integer $q\geq 3[F^+:\QQ]$ as in \cite[Proposition 4.4]{MR2979825}.
We introduce the ring \[R_\infty = R^{\textrm{loc}}[[x_1,\ldots ,x_{q-3[F^+:\QQ]}]],\]
over which we construct a patched module of automorphic forms in \S \ref{construction}.

\subsubsection{Automorphic forms on definite unitary groups} \label{aut}
As in \cite[\S 2.3]{1310.0831}, we let $\UU/\mathcal{O}_{F^+}$ be a model for a definite unitary group that is quasisplit at all finite places of $F^+$ and splits over $F$.
For each place $v$ of $F^+$ that splits as $v = ww^c$ in $F$, let $\cO_v$ and $\cO_w$ denote $\cO_{F^+_v}$ and $\cO_{F_w}$, respectively.
Let $\iota_w: \UU(\cO_v) \isom \GL_3(\cO_w)$ be the natural isomorphism.
For $r$ as in \S \ref{galrep}, one defines a corresponding ideal $\lambda \subset \TT$, where $\TT = \TT^{S_p,\textrm{univ}}$ is defined in \cite[\S 2.3]{1310.0831}.
%The ideal $\lambda$ is the Hecke eigensystem corresponding to $r^\vee$.

For a compact open subgroup $U \subset \UU(\AA^\infty_{F^+})$ and a $\cO_E[U]$-module $V$, we define the $\cO_E$-module of algebraic modular forms $S(U,V)$ to be the space of functions
\[f: \UU(F^+) \backslash \UU(\AA_{F^+}^\infty) \rarrow V\]
such that $f(gu) = u^{-1}f(g)$ for all $g\in \UU(\AA_{F^+}^\infty)$ and $u\in U$.

For $m\in \NN$, let $U_m = \prod_v U_{m,v} \subset \UU(\AA_{F^+}^\infty)$ be the compact open subgroup where
\begin{itemize}
\item $U_{m,v}= \UU(\mathcal{O}_v)$ for all places $v$ which split in $F$ other than $v_1$ and those dividing $p$;
  \item  $U_{m,v_1}$ is the preimage of the upper triangular matrices under the map
\[\UU(\mathcal{O}_{v_1})\to \UU(k_{v_1}) \underset{\iota_{\tld{v}_1}}\isom \GL_3(k_{\tld{v}_1});\]
\item $U_{m,v}$ is  the kernel of the map
$\UU(\mathcal{O}_v)\to \UU(\mathcal{O}_v/p^m)$ for $v|p$; 
 \item $U_{m,v}$ is a hyperspecial maximal compact open subgroup of $\UU(F_v)$
    if $v$ is inert in $F$.
  \end{itemize}
\begin{remark} \label{v1}
Note that $v_1$ was chosen in \S \ref{galrep} so that for all $t \in \UU(\AA_{F^+}^\infty)$, $t^{-1} \UU(F^+) t \cap U$ does not contain an element of order $p$, necessary for Theorem \ref{patchedmod}(\ref{proj}).
\end{remark}

Let $\rbar$, $\Sigma$, $\widetilde{v}_1$, and $\tau_{\tld{v}}$ be as in \S \ref{galrep}.
For $v\in \Sigma$, denote by $\sigma(\tau_v)$ the $\UU(\mathcal{O}_v) \underset{\iota_{\tld{v}}}{\isom} \GL_3(\mathcal{O}_{\tld{v}})$-representation over $E$ which is the type corresponding to $\tau_{\tld{v}}$ by results towards inertial local Langlands (see \cite[Proposition 6.5.3]{MR2656025}).
Note that $\sigma(\tau_v)$ does not depend on the choice of place ${\tld{v}}|v$ of $F$.
For each $v\in \Sigma$, let $\sigma^\circ(\tau_v)$ be a $\UU(\mathcal{O}_v)$-invariant $\mathcal{O}_E$-lattice in $\sigma(\tau_v)$, and $\overline{\sigma}(\tau_v)$ be its reduction.
Let $\sigma^\circ(\tau)$ be $\otimes_{v\in \Sigma} \sigma^\circ(\tau_v)$ and $\overline{\sigma}(\tau)$ be $\otimes_{v\in \Sigma} \overline{\sigma}(\tau_v)$.
Let $\widehat{H}^0$ and $\overline{H}^0$ be
$\varprojlim_r \varinjlim_m S(U_m,\sigma^\circ(\tau)/\pi_E^r)$ and $\varinjlim_m S(U_m,\overline{\sigma}(\tau))$,
respectively.
These are spaces of algebraic modular forms with level at $\tld{v}_1$, infinite level at places dividing $p$, and coefficients at places in $\Sigma$.

Let $\mathbb{T}^{\Sigma\cup p\cup \{v_1\},\mathrm{univ}}$ be the commutative $\cO_E$-algebra generated by the formal variables $T_w^{(j)}$ for $1\leq j \leq 3$, where $w$ is a place of $F$ lying over a place $v$ of $F^+$ away from $v_1$, $p$, and $\Sigma$ which splits as $ww^c$ in $F$.
Then $\mathbb{T}^{\mathrm{univ}}$ acts naturally on $\widehat{H}^0$ and $\overline{H}^0$ as in \cite[\S 5.3]{MR3134019}.
For a Galois representation $r:G_F \rarrow \GL_3(\cO_E)$ lifting $\rbar$, let $\lambda$ be the ideal generated by the relations given by $\det(XI-r(\Frob_w)) = X^3 - (\NN w) T_w^{(1)}X^2+(\NN w) T_w^{(2)} X - (\NN w) T_w^{(3)}$ for $w$ as above.
Similarly define the ideal $\mathfrak{m} \subset \mathbb{T}^{\mathrm{univ}}$ with respect to $\rbar$ so that $\mathfrak{m} = \lambda + (\varpi_E)$.

\begin{ax}
We say that $r$ is modular {\bf (M)} if 
\[\Hom_{\prod_{v|p}\UU(\cO_v)}(W,\widehat{H}^0)[\lambda]\neq 0\]
for some finite locally algebraic $\cO_E[\prod_{v|p}\UU(\cO_v)]$-module $W$.
\end{ax}

We now assume that $r$ is modular.
In particular, we have that $r^c \cong r^\vee \epsilon^{-2}$, $r$ is unramified outisde of $\Sigma$, and the inertial type of $r|_{G_{F_{\tld{v}}}}$ is $\tau_{\tld{v}}$ for all places $v\in \Sigma$.
Let $\pi$ and $\pibar$ be $\widehat{H}^0[\lambda]$ and $\overline{H}^0[\lambda]$.
\begin{remark} \label{surjchoice}
As explained in the section following \cite[Remark 2.9]{1310.0831}, one chooses a surjection $R_\infty \surj R_\cS^{\square_S}$.
\end{remark}
With the choice in the previous remark fixed, let $\wp \subset \mathfrak{m} \subset R_\infty$ be the inverse images of the ideals in $R_\cS^{\square_S}$ corresponding to $r$ and $\rbar$.
Though an abuse of notation, the two uses of $\mathfrak{m}$ are compatible via the usual map from the deformation ring to the Hecke algebra.
The following theorem summarizes some of the main results of \cite{1310.0831}. The proof is postponed to \S \ref{construction}.

\begin{thm} \label{patchedmod}
There exists an $R_\infty[\![\prod_{v|p} \UU(\cO_v)]\!]$-module $M_\infty$ with a compatible action of $\prod_{v|p} \UU(F^+_v)$ satisfying the following properties. For a $\prod_{v|p} \UU(\cO_v)$-representation $\otimes_{v|p} W_v$ which is finite over $\ZZ_p$, let $M_\infty(\otimes_{v|p} W_v)$ denote 
\[\Hom_{\prod_{v|p} \UU(\cO_v)}(\otimes_{v|p} W_v,M_\infty^\vee)^\vee,\] where $\cdot^\vee$ denotes the Pontrjagin dual.
Then 
\begin{enumerate}
\item $M_\infty$ is projective as an $\cO_E[\![\prod_{v|p} \UU(\cO_v)]\!]$-module; \label{proj}
\item for $\otimes_{v|p} W_v$ as above which is either a $p$-power-torsion or $p$-torsion-free locally algebraic representation, $M_\infty(\otimes_{v|p} W_v)$ is maximal Cohen-Macaulay over its scheme-theoretic support in $\mathrm{Spec} R_\infty$ or is $0$; and \label{cm}
\item $\pi = M_\infty^d[\wp]$ and $\pibar = M_\infty^\vee[\mathfrak{m}]$, where $\cdot^d$ denotes the Schikhof dual (see \cite[\S 1.8]{1310.0831}). \label{special}
\end{enumerate}
\end{thm}

\begin{remark} \label{schikhof}
By \cite[Lemma 4.14]{1310.0831}, if $\otimes_v W_v$ is $\varpi_E$-torsion-free, then 
\[M_\infty(\otimes_v W_v) \cong \Hom_K(\otimes_v W_v,M_\infty^d)^d.\]
Note that in \cite{1310.0831}, $M_\infty(\otimes_v W_v)$ is only defined for $p$-torsion-free objects and the definition with Schikhof duality is used instead.
\end{remark}

\subsection{The construction of patched modules} \label{construction}
In this section, we describe the construction of the patched module $M_\infty$ satisfying the properties in Theorem \ref{patchedmod}.
The main difference from \cite{1310.0831} is that we allow ramification away from $p$, and we patch at all places dividing $p$ simultaneously.
The necessary modifications are straightforward and the proofs are identical.

\subsubsection{Compact open subgroups and Taylor--Wiles primes} \label{compact}

%Assume {\bf (M)} with compact open subgroup $U$. Let $\Sigma$ be the set of places $v\nmid p$ in $F^+$ such that $U_v$ is not a hyperspecial compact open subgroup.
Recall the definition of $U_{m,v}$ for $m\in \NN$ and $v$ a finite place of $F^+$ from \S \ref{aut}.
We now choose Taylor--Wiles primes.
Recall that in \S \ref{galrep}, we chose an integer $q$.
As in \cite[\S 2.6]{1310.0831} (which follows \cite[\S 4]{MR2979825}), for each $N \geq 1$ we use ({\bf TW}) to choose a set $Q_N$ of $q$ places of $F^+$ satisfying the following properties: for each place $v\in Q_N$,
\begin{itemize}
\item $v$ is not in $\Sigma$, does not divide $p$, and splits in $F$;
\item $\N v \equiv 1 \mod p^N$ 
\item if $w|v$ is a place of $F$, then $\rbar|_{G_{F_w}}$ has a nonzero direct summand $\overline{\psi}_w$ which is an eigenspace on which Frobenius acts semisimply.
\end{itemize}
For each $v\in Q_N$, choose a place $\tld{v}|v$ of $F$ and let $\tld{Q}_N = \{\tld{v}:v\in Q_N\}$.
For $v \in Q_N$, we let $U_1(Q_N)_v \subset \UU(\cO_v)$ be the corresponding parahoric compact open subgroup defined in \cite[\S 5.5]{MR3134019}.
Let $U_1(Q_N)_{m,v} = U_{m,v}$ for $v \not\in Q_N$ and $U_1(Q_N)_{m,v} = U_1(Q_N)_v$ for $v\in Q_N$. We define the compact open subgroup $U_1(Q_N)_m  = \prod_v U_1(Q_N)_{m,v} \subset \UU(\AA_{F^+}^\infty)$.
Recall from \S \ref{aut} the definition of the set of places $\Sigma$ and $\sigma^\circ(\tau)$.
%Let $S_\tau(U_1(Q_N)_m)$ be $S(U_1(Q_N)_m,\sigma^\circ(\tau))$ with notation as in \S \ref{aut}.

For each $v\in Q_N$, let $R^{\psibar_{\widetilde{v}}}_{\widetilde{v}}$ be the quotient of $R^\square_{\widetilde{v}}$ defined in \cite[\S 5.5]{MR3134019}.
We let $\cS_{Q_N}$ be the deformation problem
\begin{align*}
\cS_{Q_N} = (&F/F^+,S \cup Q_N,\widetilde{S} \cup \widetilde{Q}_N, \cO_E , \rbar, \epsilon^{-2} \delta_{F/F^+}, \\
&\{ R^\square_{\widetilde{v}_1} \} \cup \{ R^{\square}_{\widetilde{v}}\}_{v|p} \cup \{ R^{\square,\tau_{\tld{v}}}_{\widetilde{v}} \}_{v \in \Sigma} \cup \{R^{\psibar_{\widetilde{v}}}_{\widetilde{v}}\}_{v\in Q_n} ).
\end{align*}
Again, we have the universal deformation ring $R_{\cS_{Q_N}}^{\textrm{univ}}$ and the universal $S$-framed deformation ring $R_{\cS_{Q_N}}^{\square_S}$.
The above properties of the set $Q_N$ crucially guarantee that $R_{\cS_{Q_N}}^{\square_S}$ is topologically generated over $R^{\mathrm{loc}}$ by $q-[F^+:\QQ]n(n-1)/2$ elements.
Let $M_{1,Q_N} = \textrm{pr}\( (S(U_1(Q_N)_{2N},\sigma^\circ(\tau)/\varpi^N)_{\mathfrak{m}_{Q_N}}\)^\vee$ where $\cdot^\vee$ denotes the Pontrjagin dual, $\textrm{pr}$ is defined in \cite[\S 2.6]{1310.0831}, and $\mathfrak{m}_{Q_N}\subset \mathbb{T}^{\Sigma\cup p\cup \{v_1\}\cup Q_N,\mathrm{univ}}$ is the maximal ideal defined as in \S \ref{aut}.
%Let $M_{1,Q_N} = \textrm{pr}\( S_\tau(U_1(Q_N)_{2N},\varpi^N)_{\mathfrak{m}_{Q_N}}\)^\vee$ where $\cdot^\vee$ denotes the Pontrjagin dual, $\textrm{pr}$ is defined in \cite[\S 2.6]{1310.0831}, and $\mathfrak{m}_{Q_N}\subset \mathbb{T}^{\Sigma\cup p\cup \{v_1\}\cup Q_N,\mathrm{univ}}$ is the maximal ideal defined as in \S \ref{aut}.
Let 
\[M_{1,Q_N}^\square = M_{1,Q_N} \otimes_{R_{\cS_{Q_N}}^{\textrm{univ}}} ~R_{\cS_{Q_N}}^{\square_S}\]
and \[S_\infty = \cO_E [[z_1,\ldots,z_{9\#S},y_1,\ldots,y_q]].\]
As in \cite[\S 2.8]{1310.0831}, we patch the spaces $M_{1,Q_N}^\square$ together using the Taylor--Wiles method to get an $R_\infty[[\prod_{v|p} \UU(\cO_v)]]$-module $M_\infty$.

\begin{proof}[Proof of Theorem \ref{patchedmod}]
For part (\ref{proj}) of Theorem \ref{patchedmod}, the proof of \cite[Proposition 2.10]{1310.0831} shows that $M_\infty$ is in fact projective as an $S_\infty[[\prod_{v|p} \UU(\cO_v)]]$-module and admits a $\prod_{v|p} \UU(F^+_v)$-action that extends the action of $\prod_{v|p} \UU(\cO_v)$.

%Assume that $\otimes_v W_v$ is $p$-torsion free and that $M_\infty(\otimes_v W_v)$ is nonzero.
%Let $R_\infty(\otimes_v W_v)$ be the ring
%\[\widehat{\otimes}_{v|p} R^{\square,W_v}_{\widetilde{v}} \widehat{\otimes} \(\widehat{\otimes}_{v\in \Sigma} R^{\square,%\tau_v}_{\widetilde{v}}\) \widehat{\otimes} R^\square_{\widetilde{v}_1} [[x_1,\ldots ,x_{q-3[F^+:\QQ]}]].\]
Part (\ref{cm}) follows from \cite[Lemma 4.18(1)]{1310.0831}.
%from the inequalities
%\begin{align*}
%\depth_{R_\infty} M_\infty(\otimes_v W_v) \geq \depth_{S_\infty} M_\infty(\otimes_v W_v) &= \dim S_\infty \\
%&= \dim R_\infty(\otimes_v W_v) \\
%&\geq\dim \supp_{R_\infty} M_\infty(\otimes_v W_v),
%\end{align*}
%where the first inequality comes from the fact that the $S_\infty$-action factors through $R_\infty$, the first equality comes from the %fact that $M_\infty$ is a projective $S_\infty$-module, the second equality follows from \cite[Theorem 3.3.4]{MR2373358}, and the %second inequality comes from Lemma 4.17(1) of \cite{1310.0831}.
%If $\otimes_v W_v$ is $p$-torsion, then there exists a locally algebraic type $\otimes_v \widetilde{W}_v$ with subquotient $\otimes_v W_v$.
%Then
%\begin{align*}
%\depth_{R_\infty} M_\infty(\otimes_v W_v) = \depth_{R_\infty} M_\infty(\otimes_v \widetilde{W}_v) - 1 &= \dim R_\infty(\otimes_v %\widetilde{W}_v) - 1\\
%&= \dim R_\infty(\otimes_v \widetilde{W}_v)/\varpi_E,
%\end{align*}
%and hence $M_\infty(\otimes_v W_v)$ is maximal Cohen-Macaulay over $R_\infty(\otimes_v \widetilde{W}_v)/\varpi_E$.
Part (\ref{special}) follows directly from the construction of $M_\infty$ (see \cite[\S 2.8]{1310.0831}).
\end{proof}

\section{Cyclicity of patched modules} \label{cyclic}
We deduce axiomatic mod $p$ multiplicity one results (see Theorems \ref{cyclicalg} and \ref{cyclicps}) for upper alcove algebraic vectors and tame principal series types.
Let $F/F^+$ and $r$ and $\rbar$ be as in \S \ref{patching}.
Throughout \S \ref{cyclic}-\ref{lattice}, assume that $p$ splits completely in $F$.
For each place $v|p$ of $F^+$, choose a place $\tld{v}|v$ of $F$.
For $\UU(\cO_v) \underset{\iota_{\tld{v}}}\isom\GL_3(\cO_{\tld{v}})$-representations $V_{\tld{v}}$, let $M_\infty(\otimes_v V_{\tld{v}})$ be defined as in Theorem \ref{patchedmod}.
Fix a place $v|p$ of $F^+$.
In what follows, we will sometimes fix a choice of $\GL_3(\cO_{\tld{v}'})$-representation $V_{\tld{v}'}$ for each $v' \neq v$ and simply write $M_\infty(V_{\tld{v}})$ for a $\GL_3(\cO_{\tld{v}})$-representation $V_{\tld{v}}$ to denote $M_\infty(V_{\tld{v}} \otimes_{v' \neq v} V_{\tld{v}'})$.

\subsection{A key isomorphism} \label{keyiso}
In this section, we prove Lemma \ref{main}, which provides the key input in proving cyclicity of patched modules for upper alcove algebraic representations.
Suppose that $\rbar$ satisfies {\bf (TW)} and $r$ satisfies {\bf (M)} from \S \ref{galrep} and \S \ref{aut}.
Fix a place $v|p$ of $F^+$.
We say that $\rbar$ is modular of $\GL_3(\cO_{\tld{v}})$-weight $\sigma$ at $\tld{v}$ if $\Hom_{\GL_3(\cO_{\tld{v}})}(\sigma, \pibar)$ is nonzero.
Note that since we define $\mathfrak{m}$ in terms of $\rbar$ rather than $\rbar^\vee$ and use $\Hom_{\GL_3(\cO_{\tld{v}})}(\sigma, \pibar)$ rather than $(\sigma\otimes \pibar)^{\GL_3(\cO_{\tld{v}})}$, our definition agrees with that of \cite{herzig,MR3079258}.
We denote the set of modular Serre weights for $\rbar$ at $\tld{v}$ by $W_{\tld{v}}(\rbar)$.
Note that the corresponding set of $\UU(\cO_v)$-representations does not depend on the choice of $\tld{v}$.
The following is an immediate corollary of Theorem \ref{patchedmod}(\ref{special}).

\begin{prop} \label{support}
%\marginpar{This is dual to the notion in EGH and HLM. Maybe that's OK. Maybe change $\sigma$ to $\sigma^\vee\in W_v(\rbar)$ or $W_v(\rbar)$ to $W_v^\vee(\rbar)$.}
A weight $\sigma$ belongs to $W_{\tld{v}}(\rbar)$ if and only if $M_\infty(\sigma) \neq 0$ for some choice of $V_{\tld{v}'}$.
\end{prop}

%As in \S \ref{heckealg}, we consider $\GL_3(\cO_{\tld{v}})$-extensions $0\rarrow \sigma' \rarrow W_1 \rarrow \sigma \rarrow 0$ and $0 \rarrow \sigma \rarrow W_2 \rarrow \sigma' \rarrow 0$ such that $W_1 \inj \Ind_{\widetilde{P}_\mu}^K \sigma'_{N_\mu}$ for some minuscule coweight $\mu$.
We introduce the following axiom.

\begin{ax}
We say that $\rbar$ satisfies weight elimination {\bf (WE)} for $\mu$ at $\tld{v}$ if $\Ind_{\widetilde{P}_\mu}^K \sigma'_{N_\mu}/\overline{W}$ has no Jordan--H\"older factors in $W_{\tld{v}}(\rbar)$.
\end{ax}

\begin{example}\label{ex:we}
Suppose that $r|_{F_{\tld{v}}}$ is irreducible with $\rbar|_{I_{\tld{v}}} \cong \psi \oplus \psi^p \oplus \psi^{p^2}$ where $\psi \cong \omega_3^{-(a_{\tld{v}}+1)-p (b_{\tld{v}}+1) - p^2(c_{\tld{v}}+1)}$.
Then, $\rbar$ satisfies {\bf (WE)} for $\mu_1$ at $\tld{v}$ by \cite[Theorem 7.5.5]{MR3079258} (see \cite[Theorem 5.1 and Conjecture 6.6]{herzigthesis} for a formula for $W^?(\rbar)$).
\end{example}

\begin{remark} \label{eliminate}
By Proposition \ref{support} and Theorem \ref{patchedmod}(\ref{proj}), if $\rbar$ satisfies {\bf (WE)} for $\mu$ at ${\tld{v}}$ then $M_\infty(\Ind_{\widetilde{P}_\mu}^K \sigma'_{N_\mu}/\overline{W}) = 0$ and $M_\infty(\ker(\Ind_{\widetilde{P}_{-\mu}}^K \sigma'^{N_{-\mu}} \rarrow \overline{W}^0)) = 0$ for all choices of $V_{v'}$.
\end{remark}

An element of $\mathcal{H}(\overline{W},\overline{W}^0)$ induces an $R_\infty$-homomorphism 
\[\Hom_{G\times \prod_{v'\neq v}\UU(\cO_{v'})}(\ind_K^G \overline{W}^0 \otimes \otimes_{v'\neq v} V_{\tld{v}'},M_\infty^\vee) \rarrow \Hom_{G\times \prod_{v'\neq v}\UU(\cO_{v'})}(\ind_K^G \overline{W} \otimes \otimes_{v'\neq v} V_{\tld{v}'},M_\infty^\vee),\]
and by duality an element of $\Hom_{R_\infty}(M_\infty(\overline{W}),M_\infty(\overline{W}^0))$.
This defines a ring homomorphism $\mathcal{H}(\overline{W},\overline{W}^0) \rarrow \Hom_{R_\infty}(M_\infty(\overline{W}),M_\infty(\overline{W}^0))$.
It is here that we use that the $R_\infty[\![\prod_{v'|p} \UU(\cO_{v'})]\!]$-module $M_\infty$ in Theorem \ref{patchedmod} has a compatible $\prod_{v'|p} \UU(F^+_{v'})$-action.
By abuse of notation, we also denote the image of $T'$ in $\Hom_{R_\infty}(M_\infty(W_1),M_\infty(W_2))$ by $T'$.
%The map $T': \ind_K^G W_1 \rarrow \ind_K^G W_2$ from \S \ref{heckered} induces a map 
%\[\Hom_G(\ind_K^G W_2,M_\infty^\vee) \rarrow \Hom_G(\ind_K^G W_1,M_\infty^\vee)\] and hence by Frobenius reciprocity and duality, a map $T':M_\infty(W_1) \rarrow M_\infty(W_2)$.

\begin{lem} \label{main}
Suppose that $\rbar$ satisfies {\bf (WE)} for $\mu$ at ${\tld{v}}$.
Then the $R_\infty$-module homomorphism $T': M_\infty(\overline{W}) \rarrow M_\infty(\overline{W}^0)$ is an isomorphism.
\begin{proof}
By Proposition \ref{extensioncycling} and Theorem \ref{patchedmod}, the map $T': M_\infty(\overline{W}) \rarrow M_\infty(\overline{W}^0)$ can be written as the composition
\[ M_\infty(\overline{W}) \rarrow M_\infty(\Ind_{\widetilde{P}_\mu}^K \sigma_{N_\mu}) \rarrow M_\infty(\Ind_{\widetilde{P}_{-\mu}}^K \sigma^{N_{-\mu}}) \rarrow M_\infty(\overline{W}^0).\]
By Frobenius reciprocity and Proposition \ref{weightcycling}, the second morphism is an isomorphism.
By the exactness of $M_\infty(\cdot)$ from Theorem \ref{patchedmod}(\ref{proj}) and Remark \ref{eliminate}, the first and third morphisms are isomorphisms.
\end{proof}
\end{lem}

\subsection{Multiplicity one for extensions}

From Lemma \ref{main}, we deduce cyclicity of certain patched modules, extending the method of \cite{MR1440309} and \cite{math/0602606}.
Recall that $\tau_{\tld{v}}$ is the inertial type of $r$ at $\tld{v}$ for each place $v\in \Sigma$.

\begin{ax}
We say that $r$ has regular type {\bf (RT)} if $\widehat{\otimes}_{v\in \Sigma} R^{\square,\tau_{\tld{v}}}_{\widetilde{v}}$ is a regular local ring.
\end{ax}

\begin{remark}
We note that the axiom {\bf (RT)} is satisfied if $r$ is minimally ramified in the sense of \cite[Definition 2.4.14]{MR2470687} by \cite[Corollary 2.4.21]{MR2470687}.
\end{remark}

Suppose that for each place $v|p$ of $F$, $\sigma_{\tld{v}}$ is a Serre weight $F(x_{\tld{v}}, y_{\tld{v}}, z_{\tld{v}})$ with $x_{\tld{v}}\geq y_{\tld{v}}\geq z_{\tld{v}}$ and $x_{\tld{v}}-z_{\tld{v}} < p-3$. The following theorem is a well-known consequence of the method of Diamond and Fujiwara.

\begin{thm} \label{multfl}
Assume that $r$ satisfies {\bf(RT)}.
Then $M_\infty(\otimes_v \sigma_{\tld{v}})$ is cyclic over $R_\infty$ or $0$.
\begin{proof}
Let $R_\infty(\otimes_v W(x_{\tld{v}}, y_{\tld{v}}, z_{\tld{v}}))'$ be the quotient 
\[\widehat{\otimes}_{v|p} R^{\psi_{\tld{v}},\square}_{\widetilde{v}} \widehat{\otimes} \(\widehat{\otimes}_{v\in \Sigma} R^{\square,\tau_{\tld{v}}}_{\widetilde{v}}\) \widehat{\otimes} R^\square_{\widetilde{v}_1}[[x_1,\ldots ,x_{q-3[F^+:\QQ]}]]\]
of $R_\infty$, where $R^{\psi_{\tld{v}},\square}_{\widetilde{v}}$ is the crystalline lifting ring of $\rbar|_{G_{F_{\tld{v}}}}$ of Hodge--Tate weights $(x_{\tld{v}}+2,y_{\tld{v}}+1,z_{\tld{v}})$.
Since the Hodge--Tate weights $(x_{\tld{v}}+2,y_{\tld{v}}+1,z_{\tld{v}})$ are in the Fontaine-Laffaille range, $R^{\psi_{\tld{v}},\square}_{\widetilde{v}}$ is formally smooth over $\cO$ by \cite[Lemmas 2.4.27 and 2.4.28]{MR2470687}.
By the choice of $v_1$, $R^\square_{\widetilde{v}_1}$ is formally smooth over $\cO$ by \cite[Lemma 2.5]{MR2470687}.
Finally, $R^{\square,\tau_{\tld{v}}}_{\widetilde{v}}$ is formally smooth over $\cO$ for all $v\in \Sigma$ by {\bf (RT)}.
We conclude that $R_\infty(\otimes_v W(x_{\tld{v}}, y_{\tld{v}}, z_{\tld{v}}))'$ is formally smooth over $\cO$.

Assume that $M_\infty(\otimes_v W(x_{\tld{v}},y_{\tld{v}},z_{\tld{v}}))$ is nonzero.
Then by the proof of \cite[Lemma 4.18]{1310.0831}, it is maximal Cohen--Macaulay over $R_\infty(\otimes_v W(x_{\tld{v}}, y_{\tld{v}}, z_{\tld{v}}))'$.
By \cite[Theorem 2.1]{MR1440309}, $M_\infty(\otimes_v W(x_{\tld{v}},y_{\tld{v}},z_{\tld{v}}))$ is a free module over $R_\infty(\otimes_v W(x_{\tld{v}}, y_{\tld{v}}, z_{\tld{v}}))'$.
By the proof of \cite[Lemma 4.18(1)]{1310.0831}, $M_\infty(\otimes_v W(x_{\tld{v}},y_{\tld{v}},z_{\tld{v}}))[1/p]$ is locally free of rank at most one over $R_\infty(\otimes_v W(x_{\tld{v}}, y_{\tld{v}}, z_{\tld{v}}))'[1/p]$.
Thus $M_\infty(\otimes_v W(x_{\tld{v}},y_{\tld{v}},z_{\tld{v}}))$, and therefore its reduction $M_\infty(\otimes_v \sigma_{\tld{v}})$, is free of rank one over its support.
\end{proof}
\end{thm}

%Let $W_1$ and $W_2$ be as in the beginning of \S \ref{heckealg}.
We fix a place $v|p$ of $F^+$ and continue to use the notation of \S \ref{keyiso}.
Suppose that $V_{\tld{v}'}$ is $p$-torsion free for all $v' \neq v$.

\begin{lem} \label{extcyclic}
Suppose that $\rbar$ satisfies {\bf (WE)} for $\mu$ at $\tld{v}$ for some minuscule coweight $\mu$.
Let $R' \subset \End_{R_\infty}(M_\infty(W^0))$ be a commutative local $R_\infty$-algebra which commutes with the action of $T$, and let $R$ be $R'[T]$.
If $M_\infty(\sigma)$ is nonzero, then $R$ is a local ring.
\begin{proof}
By the definition of $T$ in \S \ref{heckealg}, we have the commutative diagram
\begin{center}
\begin{tikzcd}
0 \arrow{r}{} &M_\infty(\sigma') \arrow{r}{} \ar[-, double equal sign distance]{d}{} &M_\infty(\overline{W}) \arrow{r}{} \arrow{d}{T'} &M_\infty(\sigma) \arrow{r}{} \arrow{d}{} &0 \\
0 \arrow{r}{} &M_\infty(\sigma') \arrow{r}{} &M_\infty(\overline{W}^0) \arrow{r}{} &M_\infty(\overline{W}^0)/T \arrow{r}{} &0,
\end{tikzcd}
\end{center}
where the map $M_\infty(\sigma') \rarrow M_\infty(\overline{W}^0)$ is defined to be the composition of the map $M_\infty(\sigma') \rarrow M_\infty(\overline{W})$ and $T'$.
The middle vertical arrow is an isomorphism by Lemma \ref{main}, and hence so is the right vertical arrow. The module $M_\infty(\sigma)$ is cyclic and nonzero by assumption.
Let $\mathfrak{a}= (z_1,\ldots,z_{9\# S},y_1,\ldots ,y_q)$ be the augmentation ideal of $S_\infty$ in the notation of \S \ref{construction}.
Then $M_\infty(W^0)/\mathfrak{a}$ is $p$-torsion free by \cite[Corollary 2.11]{1310.0831}.
By definition of $T$ and \cite[Theorem 1.2]{MR2954105}, $T$ acts on $(M_\infty(W^0)/\mathfrak{a}) \otimes_{\cO_E} E$ semisimply.
Since $M_\infty({W}^0)/(\mathfrak{a},T)$ is nonzero, each eigenvalue of the action of $T$ on $M_\infty({W}^0)/\mathfrak{a} \otimes_{\cO_E} E$ must have positive valuation.
We conclude that the image of $T$ in $R/\mathfrak{a}$ must be in the Jacobson radical.
Since $R$ is finite over $R_\infty$ (because $\End_{R_\infty}(M_\infty)$ is), any maximal ideal of $R$ must contain $\mathfrak{a}$.
Thus, $T$ is in the Jacobson radical of $R$, and $R$ is local.
\end{proof}
\end{lem}

\subsection{Multiplicity one for locally algebraic types} \label{multloc}

In this section, we prove some mod $p$ multiplicity one results for upper alcove algebraic vectors and principal series types.
For $v'|p$, let $a_{\tld{v}'}, b_{\tld{v}'},$ and $c_{\tld{v}'}$ be integers such that $a_{\tld{v}'} > b_{\tld{v}'} > c_{\tld{v}'}$ and $a_{\tld{v}'} - c_{\tld{v}'} < p-1$.
Let $\sigma_{\tld{v}'} = F(a_{\tld{v}'}-1, b_{\tld{v}'},c_{\tld{v}'}+1)$ and $\sigma'_{\tld{v}'} = F(c_{\tld{v}'}+p-1, b_{\tld{v}'},a_{\tld{v}'}-p+1)$.

\begin{ax}
We say that $\rbar$ is modular of weight $\otimes_{v'} \sigma_{\tld{v}'}$ {\bf (M')} if \[\Hom_{\prod_{v'|p}\UU(\cO_{v'})}(\otimes_{v'} \sigma_{\tld{v}'},\pibar) \neq 0.\]
\end{ax}

\noindent By Theorem \ref{patchedmod}(\ref{special}), {\bf (M')} holds if and only if $M_\infty(\otimes_{v'} \sigma_{\tld{v}'}) \neq 0$.

\subsubsection{The case of algebraic vectors}
We continue to denote by $v$ a fixed place of $F^+$ dividing $p$ and continue to use the notation of \S \ref{keyiso}.
%Let $W'_{v'} = W$ or $W^0$.
%Let $a_{\tld{v}'}, b_{\tld{v}'},$ and $c_{\tld{v}'}$ be integers such that $a_{\tld{v}'} > b_{\tld{v}'} > c_{\tld{v}'} >1$ and $a_{\tld{v}'} - c_{\tld{v}'} < p$.
Define $W_{\tld{v}'}$, $W^0_{\tld{v}'}$, $T_{\tld{v}'}$, etc.~as in \S \ref{algvect}.
Let $R^\mathrm{alg} \subset \End_{R_\infty}(M_\infty(\otimes_{v'} W^0_{\tld{v}'}))$ be the subring generated by the images of $R_\infty$ and $T_{\tld{v}'}$.
%Note that the image of $R^\mathrm{alg}$ in $\End_{R_\infty}(M_\infty(\otimes_{v'} \overline{W}'_{v'}))$ contains $T_{\overline{W}'}$ for every place $v'|p$.

\begin{thm} \label{cyclicalg}
For each place $v'|p$, let $\mu_{\tld{v}'}$ be some minuscule coweight.
Suppose that $\rbar$ satisfies {\bf(M')}, {\bf(RT)}, and {\bf (WE)} for $\mu_{\tld{v}'}$ at $\tld{v}'$.
Then $R^\mathrm{alg}$ is a local ring and $M_\infty(\otimes_{v'} W^0_{\tld{v}'})$ is free of rank one over $R^\mathrm{alg}$.
\begin{proof}
By changing $R'$ and $R$ as necessary, Lemma \ref{extcyclic} inductively shows that $R^\mathrm{alg}$ is a local ring.
Moreover, by the proof of Lemma \ref{extcyclic}, $M_\infty(\otimes_{v'} W^0_{\tld{v}'}/(\varpi_E,T_{\tld{v}'})_{\tld{v}'})$ is isomorphic to $M_\infty(\otimes_{v'} \sigma_{\tld{v}'})$.
By Nakayama's lemma and Theorem \ref{multfl}, $M_\infty(\otimes_{v'} W^0_{\tld{v}'})$ is a cyclic $R^\mathrm{alg}$-module.
By definition, it is also a faithful $R^\mathrm{alg}$-module, and is thus free of rank one.
\end{proof}
\end{thm}

\subsubsection{The case of principal series types}\label{subsec:pstype}
For places $v'|p$ of $F^+$, let $\tau_{\tld{v}'} = \ind_I^K \chi_{\tld{v}'}$, where $\chi_{\tld{v}'} = \eta^{a_{\tld{v}'}} \otimes \eta^{b_{\tld{v}'}} \otimes \eta^{c_{\tld{v}'}}$.
% and $a_{v'}, b_{v'},$ and $c_{v'}$ are integers such that $a_{\tld{v}'} > b_{\tld{v}'} > c_{\tld{v}'}$ and $a_{\tld{v}'} - c_{\tld{v}'} < p$.
So $\tau_{\tld{v}'}$ is a tame principal series $\GL_3(\cO_{\tld{v}'})$-type.
\begin{ax}
We say that $\rbar$ satisfies weight elimination {\bf (WE)} for $\tau_{\tld{v}}$ and $\tld{v}$ if %$F(a_{\tld{v}}-1,b_{\tld{v}},c_{\tld{v}}+1) \in W_{\tld{v}}(\rbar) \cap \mathrm{JH}(\taubar_{\tld{v}})$ and
\begin{align*}
W_{\tld{v}}(\rbar) \cap \mathrm{JH}(\taubar_{\tld{v}}) \subset \{&F(a_{\tld{v}}-1,b_{\tld{v}},c_{\tld{v}}+1), F(b_{\tld{v}},c_{\tld{v}},a_{\tld{v}}-p+1), \\
&F(b_{\tld{v}}+p-1,a_{\tld{v}},c_{\tld{v}}), F(c_{\tld{v}}+p-1,b_{\tld{v}},a_{\tld{v}}-p+1)\}.
\end{align*}
\end{ax}

%Suppose that $\rbar$ satisfies {\bf (WE)} for $\tau_{\tld{v}'}$ for all places $v'|p$ of $F^+$.

\begin{lem} \label{lem:cyclicps}
Suppose that $\rbar$ satisfies {\bf (M')}, {\bf (RT)}, and {\bf (WE)} for $\taubar_{{\tld{v}}}$ and $\tld{v}$ for some $v|p$.
Let $R'\subset \End_{R_\infty}(M_\infty(\tau^{s_2}))$ be a commutative local $R_\infty$-algebra which commutes with the actions of $U_1$ and $U_2$, and let $R$ be $R'[U_1,U_2]$.
If $M_\infty(\sigma)$ is nonzero, then $R$ is a local ring.
\end{lem}
\begin{proof}
The proof is similar to Lemma \ref{extcyclic}.
By Proposition \ref{uop}, we have isomorphisms
\[M_\infty(\tau^{s_2}_{\tld{v}}/\iota_{s_1}^{s_2}(\tau^{s_1}_{{\tld{v}}})) \isom M_\infty(\tau^{s_2}_{{\tld{v}}})/U_{2,{\tld{v}}}\]
and
\[M_\infty(\tau^{s_2}_{{\tld{v}}}/\iota_{w_0}^{s_2}(\tau^{w_0}_{{\tld{v}}})) \isom M_\infty(\tau^{s_2}_{\tld{v}})/U_{1,{\tld{v}}}.\]
The left sides are nonzero since $M_\infty(\sigma_{\tld{v}})$ is.
Replacing $T$ in the proof of Lemma \ref{extcyclic} by $U_1$ and $U_2$, we see that $R$ is local.
\end{proof}

Let 
$R^\mathrm{ps} \subset \End_{R_\infty}(M_\infty(\otimes_{v'| p} \tau_{{\tld{v}}'}^{s_2}))$
denote the subring generated by the image of $R_\infty$ and the Hecke operators $U_{2,{\tld{v}}'}$ and $U_{1,{\tld{v}}'}$ for all $v'|p$.

\begin{thm} \label{cyclicps}
Assume that $\rbar$ satisfies {\bf (M')}, {\bf (RT)}, and {\bf (WE)} for $\taubar_{{\tld{v}}'}$ and $\tld{v}'$ for all $v'|p$.
Then $R^\mathrm{ps}$ is a local ring and the $R^\mathrm{ps}$-module $M_\infty(\otimes_{v'|p} \tau^{s_2}_{{\tld{v}}'})$ is free of rank one.
\begin{proof}
By changing $R'$ and $R$ as necessary, Lemma \ref{lem:cyclicps} inductively shows that $R^\mathrm{ps}$ is a local ring.
By {\bf (WE)} for $\tau_{\tld{v}'}$ and $\tld{v}'$ for all $v'|p$ and Proposition \ref{submodule}, the natural inclusion 
\[M_\infty(\otimes_{v'|p}\sigma_{\tld{v}'}) \rarrow M_\infty(\otimes_{v'|p} \taubar^{s_2}_{\tld{v}'}/(\iota_{s_1}^{s_2}(\taubar^{s_1}_{\tld{v}'})+\iota_{w_0}^{s_2}(\taubar^{w_0}_{\tld{v}'}))) \cong M_\infty(\otimes_{v'|p} \taubar^{s_2}_{\tld{v}'})/(U_{1,\tld{v}'},U_{2,\tld{v}'})_{v'|p}\]
is an isomorphism.
By Nakayama's lemma and Theorem \ref{multfl}, $M_\infty(\otimes_{v'|p} \tau^{s_2}_{{\tld{v}}'})$ is a cyclic $R^\mathrm{ps}$-module.
Since it is also a faithful $R^\mathrm{ps}$-module, it is free of rank one.
\end{proof}
\end{thm}

\subsubsection{An application of cyclicity}

In this subsection, we deduce a lemma which allows one to use multiplicity one to relate the reduction of lattices to modularity of Serre weights.
We again fix a place $v|p$ of $F^+$ and continue to use the notation of \S \ref{keyiso}.

\begin{lem} \label{red}
Fix a place $v|p$ and let $V_{\tld{v}'}$ be $\cO_E[\GL_3(\cO_{\tld{v}})]$-modules for $v'\neq v$.
Let $\tau$ be a finite locally algebraic $E[\GL_3(\cO_{\tld{v}})]$-module, and let $\tau_1 \subset \tau$ be an $\cO_E$ lattice with irreducible cosocle $\sigma$ where $M_\infty(\sigma) \neq 0$.
Let $R$ be a local $R_\infty$-algebra which acts on $M_\infty(\tau_1)$ extending the action of $R_\infty$.
Suppose that $M_\infty(\tau_1)$ is a cyclic $R$-module and $\tau_2 \subsetneq \tau_1$ is a sublattice.
Then $M_\infty(\tau_2) \subset \mathfrak{m}_R M_\infty(\tau_1)$.
\begin{proof}
The inclusion $\tau_2 \inj \tau_1$ and surjection $\tau_1 \surj \sigma$, whose composition is $0$, induce the diagram
\begin{center}
\begin{tikzcd}
M_\infty(\tau_2) \arrow{r}{} &M_\infty(\tau_1) \arrow{r}{} \arrow{d}{}
&M_\infty(\sigma) \arrow{d}{}\\
&M_\infty(\tau_1)/\mathfrak{m}_R \arrow{r}{} &M_\infty(\sigma)/\mathfrak{m}_R
\end{tikzcd}
\end{center}
where the composition of the top row is $0$.
The map $M_\infty(\tau_1) \rarrow M_\infty(\sigma)$ is surjective by exactness of $M_\infty(\cdot)$, and so the bottom row is surjective.
By assumption, $M_\infty(\tau_1)/\mathfrak{m}_R$ is one dimensional and $M_\infty(\sigma) \neq 0$, and so the bottom row is an isomorphism.
We conclude that the composition $M_\infty(\tau_2) \rarrow M_\infty(\tau_1) \rarrow M_\infty(\tau_1 )/\mathfrak{m}_R$ is $0$.
\end{proof}
\end{lem}

\section{Lattices in patched modules} \label{lattice}

In this section, we deduce our main theorem, Theorems \ref{mainalg} and \ref{mainps}, on lattices in cohomology.
We again fix a place $v|p$ of $F^+$ and continue to use the notation involving $V_{{\tld{v}}'}$ from \S \ref{keyiso}.

\subsection{Lattices in patched modules for upper alcove algebraic vectors}
Recall the definitions of $W$ and $W^0$ from \S \ref{algvect}, $T$ and $T'$ from \S \ref{heckealg}, and $R^\mathrm{alg}$ from \S \ref{multloc}.

\begin{lem} \label{factoriso}
Suppose that $\rbar$ satisfies {\bf (WE)} for $\mu$ at $v$ for some minuscule coweight $\mu$. Then the $R_\infty$-module homomorphism $T': M_\infty(W) \rarrow M_\infty(W^0)$ induced by $T'\in \mathcal{H}(W,W^0)_\mu$ is an isomorphism.
\begin{proof}
The map $T': M_\infty(\overline{W}) \rarrow M_\infty(\overline{W}^0)$ is an isomorphism by Lemma \ref{main}.
By Nakayama's lemma, $T': M_\infty(W) \rarrow M_\infty(W^0)$ is surjective.

We now prove that $T': M_\infty(W) \rarrow M_\infty(W^0)$ is injective.
Let $I$ be the kernel of the map $M_\infty(W) \rarrow M_\infty(W^0)$.
Suppose that $M_\infty(\overline{W})$ and $M_\infty(\overline{W}^0)$ are nonzero and have support of dimension $d$ as $R_\infty$-modules.
Let $Z_{d}$ be the functor that assigns to an $R^\mathrm{alg}$-module the associated $d$-dimensional cycle (see \cite[Definition 2.2.5]{MR3134019}).
We have \[Z_{d}(M_\infty(W/p^n)) = Z_{d}(M_\infty(W^0/p^n))\]
by exactness of $M_\infty(\cdot)$ from Theorem \ref{patchedmod}(\ref{proj}), the fact that $\overline{W}$ and $\overline{W}^0$ have the same Jordan--H\"older factors, and additivity of $Z_{d}$ in exact sequences.

Let $I_n$ be the image of $I/p^n$ in $M_\infty(W/p^n)$. Then we have the exact sequence
\[ 0 \rarrow I_n \rarrow  M_\infty(W/p^n) \rarrow  M_\infty(W^0/p^n) \rarrow 0.\]
Again by additivity of $Z_{d}$, we see that $Z_{d}(I_n) = 0$ and so $I_n$ has support of dimension less than $d$.
Since $M_\infty(W/p^n)$ is maximal Cohen-Macaulay over its scheme-theoretic support, there are no embedded associated primes by \cite[Theorem 17.3(i)]{MR1011461}, and so $I_n=0$ and the map $I/p^n \rarrow M_\infty(W/p^n)$ must be $0$ for all $n$.
We conclude that the map $I \rarrow M_\infty(W)$ is $0$ and so $I = 0$.
\end{proof}
\end{lem}

Let $V_{v'} = W_{v'}$ or $W^0_{v'}$ for $v' \neq v$ and $v'|p$ a place of $F^+$.
With this fixed choice of $V_{v'}$, we continue to use the notation of \S \ref{keyiso}.
By Theorem \ref{cyclicalg}, we can and do fix an isomorphism $s:R^\mathrm{alg} \isom M_\infty(W^0)$, which also gives an isomorphism of $R_\infty$-modules $(T')^{-1} \circ s: R^\mathrm{alg} \isom M_\infty(W)$ by Lemma \ref{factoriso}.
This induces an $R^\mathrm{alg}$-action on $M_\infty(W)$.
Consider the element \[(M_\infty(W^0) \oplus M_\infty(W),s,(T')^{-1}\circ s) \in \mathcal{M}(R^\mathrm{alg}),\]
with $\mathcal{M}$ the functor defined in \S \ref{subsubsec:moduli}.
%This corresponds by Proposition \ref{morita} to the element $(M_\infty^\vee,s,(\otimes_v T'_v)^{-1}\circ s) \in \prod_v\mathcal{M}(R)$.
Let $ \Lambda:\ZZ_p[x,y]/(xy-p) \rarrow R^\mathrm{alg}$ be the corresponding morphism from Proposition \ref{rep}.

\begin{thm} \label{algfamily}
For each place $v'|p$, let $\mu_{\tld{v}'}$ be some minuscule coweight.
Suppose that $\rbar$ satisfies {\bf(M')}, {\bf(RT)}, and {\bf (WE)} for $\mu_{\tld{v}'}$ at $\tld{v}'$.
Then $\Lambda(y) = s^{-1}\circ T \circ s$. In other words, in the composition $\ZZ_p[x,y]/(xy-p) \rarrow R^\mathrm{alg} \rarrow \End_{R^\mathrm{alg}} (M_\infty(W^0))$, the image of $y$ is $T$.
\begin{proof}
By Proposition \ref{rep}, we have that \[\Lambda(y) = ((T')^{-1} \circ s)^{-1} \bigl(\begin{smallmatrix}
0&0\\ 1&0
\end{smallmatrix} \bigr)
s = s^{-1} \circ T' \circ i \circ s = s^{-1} \circ T \circ s.\]
\end{proof}
\end{thm}

\begin{thm} \label{mainalg}
Let $F$ be a CM field in which $p$ splits completely.
Let $F^+$ be its totally real subfield and assume that $F/F^+$ is unramified at all finite places.
Let $r: G_F\rarrow \GL_3(\mathcal{O}_E)$ be a Galois representation satisfying {\bf(M)}.
%which ramifies only at places that split in $F/F^+$.
Assume further that
\begin{itemize}
 \item $r$ satisfies {\bf (RT)};
 \item for all places $v|p$ of $F$, $r|_{F_{\tld{v}}}$ is a lattice in a crystalline representation with Hodge--Tate weights $(c_{\tld{v}}+p+1,b_{\tld{v}}+1,a_{\tld{v}}-p+1)$ with $a_{\tld{v}}-b_{\tld{v}}>6$, $b_{\tld{v}}-c_{\tld{v}}>6$, and $a_{\tld{v}}-c_{\tld{v}}<p-5$;
 \item for all places $v|p$ of $F^+$, $r|_{F_{\tld{v}}}$ is irreducible with $\rbar|_{I_{\tld{v}}} \cong \psi \oplus \psi^p \oplus \psi^{p^2}$ where $\psi \cong \omega_3^{-(a_{\tld{v}}+1)-p (b_{\tld{v}}+1) - p^2(c_{\tld{v}}+1)}$;
 \item and the reduction $\rbar:G_F \rarrow \GL_3(k_E)$ satisfies {\bf(TW)}.
\end{itemize}
Let $\lambda_{\tld{v}}$ be the trace of $\varphi$ acting on $D_{\textrm{cris}}(r|_{G_{F_{\tld{v}}}})$.
% and $W_v$ be the irreducible algebraic representation of $\GL_3(\ZZ_p)$ of highest weight $(c_v+p-1,b_v,a_v-p+1)$, $V_v = W_v \otimes_{\ZZ_p} \QQ_p$, and $W^0_v$ as in \S \ref{algvect}. 
Let $V_{\tld{v}}$, $W_{\tld{v}}$, and $W^0_{\tld{v}}$ be as in \S \ref{algvect} and $\pi$ as in \S \ref{aut}.
Then the lattice $\otimes_v V_{\tld{v}} \cap \pi \subset \otimes_v V_{\tld{v}}$ is isomorphic to $\otimes_{v|p} (i(W_{\tld{v}}^0) + p^{c_{\tld{v}}+p-1} \lambda_{\tld{v}} W_{\tld{v}}) \subset \otimes_{v|p} (W_{\tld{v}} \otimes_{\ZZ_p} E)$.
\begin{proof}
Let $\theta: R^\mathrm{alg} \rarrow E$ be the map corresponding to $r$, defined as follows.
Recall that in Remark \ref{surjchoice}, we fixed a surjection $R_\infty \surj R_\cS^{\square_S}$.
The composition of this map and the map $R_\cS^{\square_S} \rarrow E$ defining $r$ gives a map $\theta':R_\infty \rarrow E$.
Note that $M_\infty(W)/\ker(\theta')$ is nonzero for some locally algebraic module $W$ by Theorem \ref{patchedmod}(\ref{special}) and axiom {\bf (M)}.
Moreover, $W$ must be isomorphic to $\otimes_v W_{\tld{v}}$ by local-global compatibility at $p$ as in \cite[Theorem 1.2]{MR2954105}.
Hence, the $E$-point given by $\theta'$ is an $E$-point of the scheme-theoretic support of $M_\infty(\otimes_v W_{\tld{v}})$.
In other words, the map $\theta'$ factors through the faithful quotient of $R_\infty$ acting on $M_\infty(\otimes_v W_{\tld{v}})$. 
Let $\mu = \mu_1 = (0,0,1)$.
Again by \cite[Theorem 1.2]{MR2954105}, the double coset operator $\iota_{\tld{v}}^{-1}(\GL_3(\cO_{\tld{v}})t_\mu\GL_3(\cO_{\tld{v}}))$ acts by $\lambda_{\tld{v}}$.
Then $T_{\tld{v}}$ acts on $M_\infty(\otimes_v W_{\tld{v}})/\ker(\theta')$ by $p^{c_{\tld{v}}+p-1} \lambda_{\tld{v}} \in \cO$.
Thus, there is a map $\theta: R^\mathrm{alg} \rarrow E$ extending $\theta'$ such that $\theta(T_{\tld{v}})=p^{c_{\tld{v}}+p-1} \lambda_{\tld{v}}$ by \cite[Corollary 4.4.3]{MR3079258}, \cite[Theorem 1.2]{MR2954105}, Theorem \ref{patchedmod}(\ref{special}), and the unramified local Langlands correspondence.

%Note that by the inequalities in the statement of the theorem, $(a_{\tld{v}}-1,b_{\tld{v}},c_{\tld{v}}+1)$ is a strongly generic weight in the sense of \cite[Definition 6.2.2]{MR3079258}.
By \cite[Theorem A]{BLGG}, $\rbar$ satisfies {\bf (M')}.
By Example \ref{ex:we}, $\rbar$ satisfies {\bf (WE)} for $\mu_1$ at all places $v|p$.
%,Theorem 7.5.6
By Theorem \ref{algfamily}, the composition $\theta \circ \Lambda$ takes $y_{\tld{v}}$ to $p^{c_{\tld{v}}+p-1} \lambda_{\tld{v}}$.
By Theorem \ref{patchedmod}(\ref{special}) and Remark \ref{schikhof}, $\Hom_K(\otimes_v W_{\tld{v}}^0\oplus \otimes_v W_{\tld{v}},\pi)^d \cong (M_\infty(\otimes_v W_{\tld{v}}^0) \oplus M_\infty(\otimes_v W_{\tld{v}}))\otimes_\theta E$ as $\mathcal{E}$-modules.
We conclude that the $\mathcal{E}^{\textrm{op}}$-module $\Hom_K(\otimes_v W_{\tld{v}}^0\oplus \otimes_v W_{\tld{v}},\pi)$ is isomorphic to the $\mathcal{E}^{\textrm{op}}$-module corresponding to the lattice $\otimes_{v|p} (i(W_{\tld{v}}^0) + p^{c_{\tld{v}}+p-1} \lambda_{\widetilde{v}} W_{\tld{v}}) \subset \otimes_{v|p} (W_{\tld{v}} \otimes_{\ZZ_p} E)$ for each place $v|p$.
\end{proof}
\end{thm}

\begin{remark} \label{wlog}
%If $F(a_v-1,b_v,c_v+1) \notin W_v(\rbar)$, then $i: M_\infty(W_0)\rarrow M_\infty(W)$ is an isomorphism and the question of lattices is trivial.
%We now assume that $(a_v-1,b_v,c_v+1) \in W_v(\rbar)$ and that $\rbar|_{G_{F^+_v}}$ is irreducible.
%Then by Theorems 3.3.13(i) and 7.5.6 of \cite{MR3079258}, 
If instead $\rbar|_{I_{\tld{v}}} \cong \psi \oplus \psi^p \oplus \psi^{p^2}$ where $\psi \cong \omega_3^{-(a_{\tld{v}}+1)-p (c_{\tld{v}}+1) - p^2 (b_{\tld{v}}+1)}$, then a similar result holds using the Hecke operator corresponding to $\mu = (0,1,1)$.
\end{remark}

\begin{corr}\label{corr:val}
Keep the assumptions of Theorem \ref{mainalg}. Then $-c_{\tld{v}}-p+1 < \val(\lambda_{\tld{v}}) < -c_{\tld{v}}-p+2$.
\begin{proof}
Note that the first inequality follows from \cite[Corollary 4.4.3 and Proposition 4.5.2]{MR3079258}.
Let $W$ and $W^0$ be $W_{\tld{v}}$ and $W^0_{\tld{v}}$, respectively, and let $V_{\tld{v'}}$ be $W_{\tld{v}'}$ in the notation of \S \ref{cyclic}.
Lemma \ref{red} shows that $i(M_\infty(W^0)) \subset \mathfrak{m}_{R_\infty} M_\infty(W)$.
Thus $T(M_\infty(W)) \subset \mathfrak{m} M_\infty(W)$, and we deduce from the proof of Theorem \ref{mainalg} that $p^{c_{\tld{v}}+p-1}\lambda_{\tld{v}}\in \varpi \cO_E$.
On the other hand, Proposition \ref{alglat}, Lemma \ref{extcyclic}, and Lemma \ref{red} show that $pM_\infty(W) \subset \mathfrak{m} i(M_\infty(W^0))$.
We deduce that $pM_\infty(W) \subset \mathfrak{m} T(M_\infty(W))$, and that $p^{-c_{\tld{v}}-p+2}\lambda_{\tld{v}}^{-1} \in \varpi_E \cO_E$.
\end{proof}
\end{corr}

\subsection{Lattices in patched modules for principal series types} \label{latticeps}

In this section, we prove Theorem \ref{mainthm} in the case of principal series types.
We keep the notation of \S \ref{intps} and \S \ref{subsec:pstype} such as $\tau_{{\tld{v}}'}^s$ for $s\in S_3$ and $\tau^i_{{\tld{v}}'}$ for $i = 1,2,3$.
We continue to assume that $\rbar$ satisfies {\bf (M')}, {\bf (RT)}, and {\bf (WE)} for $\taubar_{{\tld{v}}'}$ and $\tld{v}'$ for all $v'|p$.

Fix a place $v|p$ of $F^+$.
Fix $s_{v'}\in S_3$ for places $v'|p$ of $F^+$ with $v' \neq v$.
From now on we denote $M_\infty(V_{\tld{v}'} \otimes \otimes_{v'\neq v} \tau^{s_{v'}}_{\tld{v}'})$ by $M_\infty(V)$ where $V$ is a $\GL_3(\cO_{\tld{v}})\cong \GL_3(\ZZ_p)$-representation.

\begin{prop} \label{tilt}
The map $\iota_{r_js_2}^{s_is_2} : M_\infty(\tau^{r_js_2}) \rarrow M_\infty(\tau^{s_is_2})$ is an isomorphism.
\begin{proof}
The map is injective by Theorem \ref{patchedmod}(\ref{proj}).
By Proposition
%\ref{cosocle} and 
\ref{submodule} and {\bf (WE)} for $\tau_v$, the cokernel of $\iota_{r_js_2}^{s_is_2}: \taubar^{r_js_2} \inj \taubar^{s_is_2}$ contains no Serre weights in $W_v(\rbar)$. 
By Proposition \ref{support}, the cokernel of
$\iota_{r_js_2}^{s_is_2}: M_\infty(\taubar^{r_js_2}) \rarrow M_\infty(\taubar^{s_is_2})$ is $0$. By Nakayama's lemma, the cokernel of $\iota_{r_js_2}^{s_is_2}: M_\infty(\tau^{r_js_2}) \inj M_\infty(\tau^{s_is_2})$ is $0$.
\end{proof}
\end{prop}

The following propositions are proved similarly.

\begin{prop} \label{shadow}
The induced map $\iota_2^{s_2}: M_\infty(\tau^2) \inj M_\infty(\tau^{s_2})$ is an isomorphism.
%\begin{proof}
%The map is injective again by Theorem \ref{patchedmod}(\ref{proj}).
%By the proof of \cite[Lemma 4.1.1]{1305.1594}, the image of $\overline{\iota}_2$ is the minimal submodule containing $F(a-1,b,c+1)$ as a Jordan--H\"older factor.
%By Proposition \ref{submodule} and {\bf (WE)} for $\tau_v$ and $\tld{v}$, we see that the cokernel of $\overline{\iota}_2: \taubar^2 \inj \taubar^{s_2}$ contains no Jordan--H\"older factors in $W_{\tld{v}}(\rbar)$.
%Again by Nakayama's lemma, the cokernel of $M_\infty(\tau^2) \inj M_\infty(\tau^{s_2})$ is $0$.
%\end{proof}
\end{prop}

For $i = 1,2,3$, let $\iota^i: \tau^{s_1}\oplus\tau^{w_0} \rarrow \tau^i$ denote the sum of $\iota_{s_1}^i$ and $\iota_{w_0}^i$.

\begin{prop} \label{othershadow}
For $i = 1,3$, the map $\iota^i: M_\infty(\tau^{s_1}\oplus\tau^{w_0}) \rarrow M_\infty(\tau^i)$ is surjective.
%\begin{proof}
%By Propositions \ref{submodule} and Proposition \ref{submodule1}, the cokernel of $\overline{\iota}^i: %\taubar^{s_1}\oplus\taubar^{w_0} \rarrow \taubar^i$ has two Jordan--H\"older factors, both of which are not in $W_{\tld{v}}(\rbar)$ by {\bf (WE)} for $\tau_{\tld{v}}$ and $\tld{v}$.
%By Proposition \ref{support}, the cokernel of $\overline{\iota}^i: M_\infty(\taubar^{s_1}\oplus\taubar^{w_0}) \rarrow %M_\infty(\taubar^i)$ is $0$.
%By Nakayama's lemma, $\iota^i$ is surjective.
%\end{proof}
\end{prop}

Using Theorem \ref{cyclicps}, we can and do choose an isomorphism $\psi^{s_2}: R^\mathrm{ps} \isom M_\infty(\tau^{s_2})$.
The compositions $\psi^{r_i{s_2}} := \nu^{r_i{s_2}}_{s_2} \circ \psi^{s_2}: R^\mathrm{ps} \isom M_\infty(\tau^{s_2}) \isom M_\infty(\tau^{r_i{s_2}})$ are also isomorphisms by Proposition \ref{normalizer}.

For each $i$, Proposition \ref{tilt} gives an isomorphism $\psi^{s_is_2} := \iota^{s_is_2}_{r_js_2} \circ \psi^{r_j{s_2}}: R^\mathrm{ps} \isom M_\infty(\tau^{r_j{s_2}}) \isom M_\infty(\tau^{s_i{s_2}})$. This also gives an isomorphism 
\[\psi^{w_0s_2} := \nu^{w_0s_2}_{s_1s_2} \circ \psi^{s_1s_2} : R^\mathrm{ps} \isom M_\infty(\tau^{w_0s_2}).\]
Through these isomorphisms, we endow $M_\infty(\otimes_{v'} \tau^{s_{v'}}_{v'})$ with the structure of an $R^{\mathrm{ps}}$-module for all choices of $(s_{v'})_{v'}$.

\begin{lem}\label{psfamily}
We have the following identities:
\begin{enumerate}
\item $(\psi^{s_2})^{-1} \circ \iota^{s_2}_{r_is_2} \circ \psi^{r_is_2}=(\psi^{s_2})^{-1}U_j\psi^{s_2}$;
\item $(\psi^{s_2})^{-1} \circ \iota^{s_2}_{s_is_2} \circ \psi^{s_is_2}=(\psi^{s_2})^{-1}U_i\psi^{s_2}$; and
\item $(\psi^{s_2})^{-1} \circ \iota^{s_2}_{w_0s_2} \circ \psi^{w_0s_2}=(\psi^{s_2})^{-1}U_1\psi^{s_2} \cdot (\psi^{s_1s_2})^{-1}U_2\psi^{s_1s_2}$.
\end{enumerate}
\begin{proof}
For (1), we note that
\[(\psi^{s_2})^{-1} \circ \iota^{s_2}_{r_is_2} \circ \psi^{r_is_2}
= (\psi^{s_2})^{-1} \circ \iota^{s_2}_{r_is_2} \circ \nu^{r_i{s_2}}_{s_2} \circ \psi^{s_2}
= (\psi^{s_2})^{-1} U_j \psi^{s_2}\]
using Proposition \ref{uop}.
For (2), Proposition \ref{intercomp} gives
\begin{align*}
(\psi^{s_2})^{-1} \circ \iota^{s_2}_{s_is_2} \circ \psi^{s_is_2}
= (\psi^{s_2})^{-1} \circ \iota^{s_2}_{s_is_2} \circ \iota^{s_is_2}_{r_js_2} \circ \psi^{r_j{s_2}}
= (\psi^{s_2})^{-1} \circ \iota^{s_2}_{r_js_2} \circ \psi^{r_js_2},
\end{align*}
which is $(\psi^{s_2})^{-1}U_i\psi^{s_2}$ by (1).
For (3),
\begin{align*}
(\psi^{s_2})^{-1} \circ \iota^{s_2}_{w_0s_2} \circ \psi^{w_0s_2}
&=
(\psi^{s_2})^{-1} \circ \iota^{s_2}_{s_1s_2} \iota_{w_0s_2}^{s_1s_2} \circ \psi^{w_0s_2}
\\
&=
(\psi^{s_2})^{-1} \circ \iota^{s_2}_{s_1s_2} \circ \psi^{s_1s_2} (\psi^{s_1s_2})^{-1} \iota_{w_0s_2}^{s_1s_2} \circ \psi^{w_0s_2}
\\
&=
(\psi^{s_2})^{-1}U_1\psi^{s_2} \cdot (\psi^{s_1s_2})^{-1}U_2\psi^{s_1s_2},
\end{align*}
where the first equality uses Proposition \ref{intercomp} and the third equality uses (2) and the fact that $(\psi^{s_1s_2})^{-1} \iota_{w_0s_2}^{s_1s_2} \circ \psi^{w_0s_2} = (\psi^{s_1s_2})^{-1}U_2\psi^{s_1s_2}$ by an argument similar to the one for (1).
\end{proof}
\end{lem}

\begin{prop} \label{localglobal}
Let $\alpha_1$, $\alpha_2$, and $\alpha_3$ be the eigenvalues of $\varphi$ acting on the $\eta^{a_v}$, $\eta^{b_v}$, and $\eta^{c_v}$-eigenspaces, respectively, of $\mathrm{WD}(r^\vee|_{G_{F_v}})^{\mathrm{F}-\mathrm{ss}}$.
Then $U_1$ and $U_2$ act on $\Hom_{\GL_3(\cO_v)}(\tau^s_v, \pi)$ by $\alpha_{s(1)}$ and $\alpha_{s(1)}\alpha_{s(2)}/p$, respectively.
\begin{proof}
This follows from the first two paragraphs of the proof of \cite[Theorem 4.5.2]{HLM}.
\end{proof}
\end{prop}

%\begin{prop} \label{sum}
%Given an $\cO_E$-lattice $L \subset \tau \otimes_{\ZZ_p} E$, $L$ is the sum of the saturations of each lattice with irreducible cosocle as described in Lemma 4.1.1 of \cite{1305.1594}.
%\begin{proof}
%The image of this sum in $L$ modulo $\varpi_E$ is $\overline{L}$ since the Jordan--H\"older factor $\sigma$ is in the image of the lattice with cosocle $\sigma$.
%By Nakayama's lemma, the image of the sum is $L$.
%\end{proof}
%\end{prop}

\begin{thm} \label{mainps}
Let $F$ be a CM field in which $p$ splits completely.
Let $F^+$ be its totally real subfield and assume that $F/F^+$ is unramified at all finite places.
Let $r: G_F\rarrow \GL_3(\mathcal{O}_E)$ be a Galois representation satisfying {\bf(M)}. Assume further that
\begin{itemize}
 \item $r$ satisfies \emph{(}{\bf RT}\emph{)};
 \item for all places $v'|p$ of $F^+$, $r|_{F_{{\tld{v}}'}}$ is a lattice in a potentially crystalline representation with Hodge--Tate weights $(0,1,2)$ with tame type $\eta^{-a_{{\tld{v}}'}} \oplus \eta^{-b_{{\tld{v}}'}} \oplus \eta^{-c_{{\tld{v}}'}}$ where $a_{\tld{v}}-b_{\tld{v}}>6$, $b_{\tld{v}}-c_{\tld{v}}>6$, and $a_{\tld{v}}-c_{\tld{v}}<p-5$;
 \item for all places $w|p$ of $F$, $\rbar|_{F_{\tld{v}}}$ is irreducible with $\rbar|_{I_{\tld{v}}} \cong \psi \oplus \psi^p \oplus \psi^{p^2}$ where $\psi \cong \omega_3^{-(a_{\tld{v}}+1)-p (b_{\tld{v}}+1)- p^2(c_{\tld{v}}+1)}$;
 \item and its reduction $\rbar:G_F \rarrow \GL_3(k_E)$ satisfies \emph{(}{\bf TW}\emph{)}.
\end{itemize}
Let $\tau_{\tld{v}'}$, $\tau^s_{\tld{v}'}$, etc.~be as in \S \ref{intps} and $\pi$ as in \S \ref{aut}.
Then the lattice $(\otimes_{v'} \tau_{\tld{v}'} \otimes_{\ZZ_p} E) \cap \pi \subset \otimes_{v'} \tau_{\tld{v}'} \otimes_{\ZZ_p} E$ is isomorphic to a tensor product of lattices with factor at $v$ given by
\begin{align*}
\tau^{s_2} &+ 
(\alpha_1\alpha_2/p)^{-1} \iota_{e}^{s_2}(\tau^{e}) + 
\alpha_2^{-1} \iota_{r_1}^{s_2}(\tau^{r_1}) \\
&+ (\alpha_1\alpha_2\alpha_3/p)^{-1} \iota_{r_2}^{s_2} (\tau^{r_2}) + 
(\alpha')^{-1} \iota_1(\tau^1) + (\alpha')^{-1} \iota_3(\tau^3),
\end{align*}
where $\alpha'$ is the one of $\alpha_1 \alpha_2/p$ and $\alpha_2$ which has smaller valuation.
\begin{proof}
Let $\theta: R^\mathrm{ps} \rarrow \cO_E$ be the map corresponding to $r$, defined as follows.
Recall that by Remark \ref{surjchoice}, we chose a surjection $R_\infty \surj R_\cS^{\square_S}$.
The composition of this map and the map $R_\cS^{\square_S} \rarrow \cO_E$ defining $r$ gives a map $\theta':R_\infty \rarrow \cO_E$.
Let $\tau_{\tld{v}}$ be defined as in \S \ref{subsec:pstype}.
The map $\theta'$ factors through the faithful quotient of $R_\infty$ acting on $M_\infty(\otimes_v \tau^{s_2}_{\tld{v}})$ since $M_\infty(\otimes_v \tau^{s_2}_{\tld{v}})/\ker(\theta')$ is nonzero by Theorem \ref{patchedmod}(\ref{special}), axiom {\bf (M)}, local-global compatibility at $p$ as in \cite[Theorem 1.2]{MR2954105}, and inertial local Langlands as in \cite[Proposition 2.4.1(ii)]{MR3079258}.
We extend $\theta'$ to $\theta$ using Proposition \ref{localglobal}.

By \cite[Lemma 4.1.2]{1305.1594}, the factor at $v$ of the lattice $(\otimes_{v'|p} \tau_{\tld{v}'} \otimes_{\ZZ_p} E) \cap \pi$ takes the form
\begin{align*}
\tau^{s_2} + 
&(\alpha^{r_1})^{-1} \iota_{r_1}^{s_2}(\tau^{r_1}) + 
(\alpha^{r_2})^{-1} \iota_{r_2}^{s_2}(\tau^{r_2}) + 
(\alpha^{s_1})^{-1} \iota_{s_1}^{s_2}(\tau^{s_1})\\ + 
&(\alpha^{w_0})^{-1} \iota_{w_0}^{s_2}(\tau^{s_2}) +
(\alpha^{e})^{-1} \iota_e^{s_2} (\tau^e) + 
\alpha' \iota_1(\tau^1) + \alpha'' \iota_2(\tau^2) + \alpha''' \iota(\tau^3)
\end{align*}
for some constants $\alpha^\bullet$.
As in the proof of \ref{mainalg}, $\rbar$ satisfies {\bf (M')} and {\bf (WE)} for $\tau_{\tld{v}}$ and $\tld{v}$ at all places $v|p$.
Thus, we choose isomorphisms $\psi^\bullet$ as before.
By Proposition \ref{patchedmod}(\ref{special}), the isomorphisms $\psi^{s}$ induce identifications $\Hom_{\prod_{v'} \UU(\cO_v')} (\tau^{s} \otimes \otimes_{v' \neq v} \tau^{s_{v'}}_{\tld{v}'},\pi) \isom \cO_E$ via $\otimes_\theta \cO_E$ and Schikhof duality.

By Lemma \ref{psfamily}, the map 
\[\cO_E \cong \Hom_{\prod_{v'} \UU(\cO_v')} (\tau^{s_2} \otimes  \otimes_{v' \neq v} \tau^{s_{v'}}_{\tld{v}'},\pi) \rarrow \Hom_{\prod_{v'} \UU(\cO_v')} (\tau^{e} \otimes \otimes_{v' \neq v} \tau^{s_{v'}}_{\tld{v}'},\pi) \cong \cO_E\]
induced by $\iota_{e}^{s_2}$ is multiplication by $\alpha_1 \alpha_2/p$; and so $\alpha^{e} = \alpha_1 \alpha_2/p$.
The values for $\alpha^{r_1}$, $\alpha^{r_2}$, $\alpha^{s_1}$, and $\alpha^{w_0}$ are obtained similarly.

We deduce that $\alpha''$ can be taken to be $1$ by Proposition \ref{shadow}.
By Proposition \ref{othershadow}, for $i = 1,3$, either the map 
\begin{equation} \label{noncyclic1}
\Hom_{\prod_{v'} \UU(\cO_v')} (\tau^i \otimes \otimes_{v' \neq v} \tau^{s_{v'}}_{\tld{v}'},\pi) \rarrow \Hom_{\prod_{v'} \UU(\cO_v')} (\tau^{s_1} \otimes \otimes_{v' \neq v} \tau^{s_{v'}}_{\tld{v}'},\pi)
\end{equation}
or the map
\begin{equation} \label{noncyclic2}
\Hom_{\prod_{v'} \UU(\cO_v')} (\tau^i \otimes \otimes_{v' \neq v} \tau^{s_{v'}}_{\tld{v}'},\pi) \rarrow \Hom_{\prod_{v'} \UU(\cO_v')} (\tau^{w_0} \otimes \otimes_{v' \neq v} \tau^{s_{\tld{v}'}}_{\tld{v}'},\pi)
\end{equation}
is an isomorphism.
As the maps $\iota_{w_0}^{s_2}$ and $\iota_{s_1}^{s_2}$ factor through $\iota^i$, we see that $\alpha'$ and $\alpha'''$ can be taken to be one of $\alpha_2$ and $\alpha_1\alpha_2/p$.
Finally, $\val(\alpha'), \val(\alpha''') \leq \min(\val(\alpha_2), \val(\alpha_1\alpha_2/p))$ since the maps in (\ref{noncyclic1}) (resp. (\ref{noncyclic2})) are given by multiplication by the quotient of $\alpha_1\alpha_2/p$ (resp. $\alpha_2$) by $\alpha'$ and $\alpha'''$, which must be elements of $\cO_E$.
We conclude that $\alpha'$ and $\alpha'''$ can be taken to be the one of $\alpha_2$ and $\alpha_1 \alpha_2/p$ which has smaller valuation.
This yields the desired formula.
\end{proof}
\end{thm}

\begin{remark}
As in Remark \ref{wlog}, the case $\rbar|_{I_{\tld{v}}} \cong \psi \oplus \psi^p \oplus \psi^{p^2}$ where $\psi \cong \omega_3^{-(a_{\tld{v}}+1)-p (c_{\tld{v}}+1)- p^2(b_{\tld{v}}+1)}$ is treated similarly.
\end{remark}

\bibliographystyle{amsalpha}

\newcommand{\etalchar}[1]{$^{#1}$}
\providecommand{\bysame}{\leavevmode\hbox to3em{\hrulefill}\thinspace}
\providecommand{\MR}{\relax\ifhmode\unskip\space\fi MR }
\providecommand{\MRhref}[2]{%
  \href{http://www.ams.org/mathscinet-getitem?mr=#1}{#2}
}
\providecommand{\href}[2]{#2}

{\footnotesize \textit{E-mail address:} {\tt le@math.toronto.edu}}

\vskip 0.2cm

 {\footnotesize \sc Department of Mathematics, The University of Toronto}

\end{document}